\documentclass[a4paper]{article}
\usepackage{eucal}
\usepackage{amsmath,amsthm,amssymb}
\usepackage{amsfonts}
\usepackage{latexsym}
\usepackage{graphicx}
\usepackage{amsrefs}
\usepackage{enumerate}
\usepackage{array}
\usepackage[all,2cell]{xy}
\UseAllTwocells

\usepackage{tikz-cd}

\usepackage{bm}
\usepackage{pb-diagram}
\usepackage{pb-xy}

\usepackage{hyperref}
\hypersetup
{
    colorlinks,	%
    citecolor=black,%
    filecolor=black,%
    linkcolor=black,%
    urlcolor=black
}

\usepackage{time}

\newcommand{\cD}{\mathcal{D}}
\newcommand{\cE}{\mathcal{E}}

\newcommand{\cP}{\mathcal{P}}

\newcommand{\cT}{\mathcal{T}}

\newcommand{\category}[1]{\mathbf{#1}}

\newcommand{\Z}{{\mathbb{Z}}}

\newcommand{\Ext}{\operatorname{Ext}}

\newcommand{\pr}{\mathrm{pr}}

\newcommand{\Triv}{\category{Triv}}
\newcommand{\Cof}{\category{Cof}}
\newcommand{\Fib}{\category{Fib}}
\newcommand{\Prj}{\category{Prj}}
\newcommand{\Inj}{\category{Inj}}
\newcommand{\Cntr}{\category{Cntr}} 
\newcommand{\ho}{\mathrm{ho}}
\newcommand{\splt}{\mathrm{split}} 

\newcommand{\la}{\langle}
\newcommand{\ra}{\rangle}

\newcommand{\set}[2]{%
\left\{#1 \,\middle|\, #2\right\}
}

\newcommand{\coequalizer}[5]{\xymatrix{
 {\displaystyle #1} 
 \ar@<1ex>[r]^-{#2} \ar@<-1ex>[r]_-{#3}
 & {\displaystyle #4} \ar[r] & {\displaystyle #5}
 }%
}
\newcommand{\equalizer}[5]{
\xymatrix{
 {\displaystyle #1} \ar[r] & {\displaystyle #2}
 \ar@<1ex>[r]^-{#3} \ar@<-1ex>[r]_-{#4}
 & {\displaystyle #5} 
 }%
}

\newcommand{\rarrow}[1]{\buildrel #1 \over \longrightarrow}

\newcommand{\lMod}[1]{{#1}\text{-}\mathbf{Mod}}

\newcommand{\lmod}[1]{{#1}\text{-}\mathbf{mod}}

\newcommand{\SemiPrj}{\category{SemiPrj}}
\newcommand{\QiPrj}{\category{QiPrj}}
\newcommand{\HoPrj}{\category{HoPrj}}
\newcommand{\Perf}{\cP\mathrm{erf}}

\newcommand{\Quism}{\category{Quism}}

\newcommand{\Ker}{\operatorname{Ker}}

\newcommand{\Coker}{\operatorname{Coker}}

\newcommand{\Ima}{\operatorname{Im}}

\theoremstyle{plain}
\newtheorem{thm}{Theorem}[section]
\newtheorem{theorem}[thm]{Theorem}

\newtheorem{proposition}[thm]{Proposition}
\newtheorem{lemma}[thm]{Lemma}

\newtheorem{corollary}[thm]{Corollary}

\theoremstyle{definition}

\newtheorem{definition}[thm]{Definition}

\newtheorem{convention}[thm]{Convention}

\newtheorem{example}[thm]{Example}

\theoremstyle{remark}

\newtheorem{remark}[thm]{Remark}

\numberwithin{equation}{section}

\ifx\undefined\bysame
\newcommand{\bysame}{\leavemode\hbox to3em{\hrulefill}\,}
\fi

\title{Cotorsion Pairs in Hopfological Algebra}
\date{}
\author{Mariko Ohara and Dai Tamaki}

\begin{document}

\maketitle

\begin{abstract}
 In an intriguing paper \cite{math/0509083} Khovanov proposed a
 generalization of homological algebra, called Hopfological algebra.
 Since then, several attempts have been made to import tools and
 techiniques from homological algebra to Hopfological algebra.
 For example, Qi \cite{1205.1814} introduced the notion of cofibrant
 objects in the category $\bm{C}_{A,H}^{H}$ of $H$-equivariant modules
 over an $H$-module algebra $A$, which is a counterpart to the category
 of modules over a dg algebra, although he did not define a model
 structure on $\bm{C}_{A,H}^{H}$. 

 In this paper, we show that there exists an Abelian model structure on
 $\bm{C}_{A,H}^{H}$ in which cofibrant objects agree with Qi's cofibrant
 objects under a slight modification.
 This is done by constructing cotorsion pairs in $\bm{C}_{A,H}^{H}$
 which form a Hovey triple in the sense of Gillespie \cite{1512.06001}. 
 This can be regarded as a Hopfological analogues of the
 work of Enochs, Jenda, and Xu \cite{Enochs-Jenda-Xu1996} and
 Avramov, Foxby, and Halperin \cite{Avramov-Foxby-Halperin}. 
 By restricting to compact cofibrant objects, we obtain a Waldhausen
 category $\Perf_{A,H}^{H}$ of perfect objects. By taking invariants of this
 Waldhausen category, such as algebraic $K$-theory, Hochschild homology,
 cyclic homology, and so on, we obtain Hopfological analogues of these
 invariants.
\end{abstract}

\tableofcontents

\section{Introduction}
\label{introduction}

Khovanov \cite{math/0509083} proposed a generalization of homological
algebra, called Hopfological algebra, based on finite dimensional Hopf
algebras. 
An important observation of Khovanov is that the existence of an
integral in a finite dimensional Hopf algebra $H$ allows us to define an
analogue of chain homotopy and homology in the category $\lMod{H}$ of
left $H$-modules, with which an analogue of homological algebra can be
developed, generalizing the fact that the category of chain complexes
can be identified with the category of $\Z$-graded modules over the
exterior Hopf algebra $\Lambda(d)$.

More generally, a differential graded algebra, dg algebra for short, $A$
is nothing else but a $\Z$-graded $\Lambda(d)$-module algebra.
Given a finite dimensional Hopf algebra $H$ over a field $k$ and a left
$H$-module algebra $A$, Khovanov proposed to study the category
$\bm{C}_{A,H}^{H}$ of $H$-equivariant left $A$-modules (see
\S\ref{module_category} for a precise definition) as a
generalization of homological algebra over a dg algebra.
Following Khovanov's proposal, Qi \cite{1205.1814} defined the derived
category of compact objects $\cD^{c}(A,H)$ and defined the Grothendieck
group $K_{0}(A,H)$ of the pair $(A,H)$ as the Grothendieck group of
$\cD^{c}(A,H)$.

The analogy between homological and Hopfological algebra can be
summarized in the following table. 
\begin{center}
 \begin{tabular}{|c|c|} \hline
  homological algebra & Hopfological algebra \\ \hline \hline
  chain complex & $H$-module \\ \hline
  chain map & $H$-module homomorphism \\ \hline
  chain homotopy & homotopy defined by an integral $\lambda$ \\ \hline
  homology $H(M)=\Ker d/\Ima d$ & homology $H(M)=M^{H}/\lambda\cdot M$ \\ \hline
  dg algebra & $H$-module algebra \\ \hline
  dg category & $H$-module category \\ \hline
  left module over a dg category  & $H$-equivariant left module over an
      $H$-module category  \\
  \hline 
\end{tabular}
\end{center}

Qi also proposed in Remark 7.17 of his paper to define and study the
higher algebraic $K$-theory of $\cD^{c}(A,H)$ by using the method
introduced by Thomason and Trobaugh \cite{Thomason-Trobaugh1990}, in
which the algebraic $K$-theory of derived categories (in the sense of
usual homological algebra) is defined.
However, the definition of $\cD^{c}(A,H)$ is quite different from the
usual definition of the derived category of a dg algebra or a dg
category, since Qi did not use homology.

Recall that the algebraic $K$-theory of a dg category $A$ is defined as
the Waldhausen $K$-theory of the category $\Perf_{A}$ of compact
cofibrant objects in the category $\lMod{A}$ of left $A$-modules by
introducing a model structure on $\lMod{A}$. See To{\"e}n's lecture note
\cite{Toen2011}, for example.
For an $H$-module algebra $A$, Qi introduced the notion of cofibrant
objects in $\bm{C}_{A,H}^{H}$ and proved the existence of a functorial
cofibrant replacement functor without introducing a model structure.

This aim of this paper is to construct a model structure on
$\bm{C}_{A,H}^{H}$ in which cofibrant objects agree with Qi's cofibrant
objects under a slight modification. The modification is needed because
of the difference of weak equivalences. We use isomorphisms of homology,
while Qi used isomorphisms in the stable category of $H$-modules.

Since $\bm{C}_{A,H}^{H}$ is an Abelian category, we should make use of
Hovey's theory of Abelian model structures \cite{Hovey2002-2}. By using
the terminology of Gillespie \cite{1512.06001}, given an Abelian category
$\bm{A}$, Hovey found a one-to-one correspondence between Abelian model
structures on $\bm{A}$ and Hovey triples.
Recall that a Hovey triple in $\bm{A}$ is a triple $(\Cof,\Triv,\Fib)$
of subcategories such that both $(\Cof,\Triv\cap \Fib)$ and
$(\Cof\cap\Triv,\Fib)$ are complete cotorsion pairs and $\Triv$ is a thick
subcategory. Recall also that cotorsion pairs are defined in terms of
the orthogonality with respect to the biadditive functor
$\Ext^{1}_{\bm{A}}(-,-)$. See \S\ref{model_category} for details.

In the case of the category of chain complexes, and, more generally,
in the category of left modules over a dg algebra, the corresponding
orthogonality has been studied by Enochs, Jenda, and Xu 
\cite{Enochs-Jenda-Xu1996} and Avramov, Foxby, and Halperin
\cite{Avramov-Foxby-Halperin} in detail.
As is stated in Hovey's paper, their results lead to a Hovey triple
which gives rise to the standard model structure on such categories, in
which cofibrant objects are semiprojective modules.

We introduce the notions of $\Sigma$-semiprojective
objects (Definition \ref{def:equivariant_projectives}) and
$\Sigma$-quasi-isomorphisms (Definition 
\ref{def:quism_for_H-module}) in 
$\bm{C}_{A,H}^{H}$ that are 
analogues of semiprojective modules and quasi-isomorphisms in the dg
context and show that they form a part of a Hovey triple. 

\begin{theorem}
 \label{main_theorem:model_structure}
 Let $H$ be a finite dimensional non-semisimple Hopf algebra over a
 field and $A$ a left $H$-module category. 
 Denote the full subcategories of $\bm{C}_{A,H}^{H}$ consisting of
 $\Sigma$-semiprojective objects and of those objects that are
 $\Sigma$-quasi-isomorphic to $0$ by $\SemiPrj_{\Sigma}$ and 
 $\Triv_{A,H}^{\Sigma}$, respectively.

 Then the triple
 $(\SemiPrj_{\Sigma},\Triv_{A,H}^{\Sigma},\bm{C}_{A,H}^{H})$ is a Hovey
 triple and thus defines an Abelian model structure on
 $\bm{C}_{A,H}^{H}$ in which weak equivalences are
 $\Sigma$-quasi-isomorphisms, cofibrant objects are
 $\Sigma$-semiprojective modules, and all objects are fibrant.
\end{theorem}

By analogy, we call compact cofibrant objects in our model category
\emph{perfect objects}.  
The full subcategory $\Perf_{A,H}^{H}$ of perfect objects 
has a structure of a Waldhausen category.
We propose to call the algebraic $K$-theory of this Waldhausen category
the algebraic $K$-theory of the pair $(A,H)$. 
Besides algebraic $K$-theory, the Waldhausen category $\Perf_{A,H}^{H}$
allows us to extend invariants of dg categories, such as Hochschild
homology, cyclic homology, and trace maps
between them. Their properties will be studied in a sequel to this paper.

Recall that there is another approach to the algebraic $K$-theory of dg 
categories, as is described in \S5.2 of Keller's article
\cite{math/0601185}. 
Given a dg category $A$, Keller defines the algebraic $K$-theory of $A$
as the $K$-theory of the Waldhausen category of compact $A$-modules,
whose cofibrations are morphisms $i:L\to M$ which admits a retraction in
the category of $\underline{A}$-modules, where $\underline{A}$ is the
underlying (graded) linear category of $A$.

We may also define a Hopfological analogue of this construction by using
a cotorsion pair.
For a left $H$-module category $A$, we define a structure of an exact
category on $\bm{C}_{A,H}^{H}$ by declaring $\underline{A}$-split
extensions as exact sequences, where $\underline{A}$ is the linear
category obtained from $A$ by forgetting the $H$-action.
Let us denote this exact category by $\bm{C}_{A,H}^{H,\splt}$.
We show that the pair $(\bm{C}_{A,H}^{H,\splt},\Cntr_{A,H}^{H})$ is a
complete cotorsion pair, where $\Cntr_{A,H}^{H}$ is the full subcategory
consisting of objects that are homotopy equivalent to $0$ in the
category of left $H$-modules. See Proposition
\ref{cotorsion_pair_from_contractible_objects}.  
Although this cotorsion pair is not a part of a Hovey triple, a recent
work of Sarazola's \cite{1911.00613} allows us to construct a Waldhausen
category from this cotorsion pair and $\Triv_{A,H}^{\Sigma}$.
The Waldhausen subcategory of compact objects is another choice for
defining algebraic $K$-theory of $(A,H)$. 
We note that this is closer to Qi's approach to the
Grothendieck group of $\cD^{c}(A,H)$.

Finally we remark that Kaygun and Khalkhali \cite{math/0606341}
introduced another kind of ``projective'' modules for an $H$-module
algebra $A$, called $H$-equivariantly projective $A$-modules in their
paper. 
Their purpose is to define the Hopf-cyclic homology of $A$ by using the
exact category of $H$-equivariantly projective $A$-modules.
We may use this exact category to define an algebraic $K$-theory.
From this point of view, however, the action of $H$ is regarded as a
generalization of group actions, while, in Hopfological algebra, the
action of $H$ is a generalization of differentials.
Thus the $K$-theory obtained from $H$-equivariant projective $A$-modules 
should be regarded as a generalization of Thomason's equivariant
$K$-theory \cite{Thomason1987} and is different from ours.

\subsection*{Organization}

The rest of this paper consists of two sections.

\begin{itemize}
 \item \S\ref{preliminary} is preliminary.
       Notations and conventions used in this paper are listed in
       \S\ref{notation}. We recall basic properties of the category
       of $H$-equivariant $A$-modules $\bm{C}_{A,H}^{H}$ in
       \S\ref{module_category}. 
       Basics ideas in Hopfological algebra are recalled in 
       \S\ref{homological_algebra_in_stable_module_category}.
       And \S\ref{model_category} is a brief summary of Hovey's theory
       of Abelian model structures used in this paper. 

 \item \S\ref{cotorsion_pair} is the main part. In \S\ref{projectives},
       the notion of $\Sigma$-semiprojective modules and related
       structures are introduced and studied, with which  
       Theorem \ref{main_theorem:model_structure} is proved in
       \S\ref{orthogonality} by studying the orthogonality in
       the category $\bm{C}_{A,H}^{H}$.   
\end{itemize}

\subsection*{Acknowledgements}

The second author is supported by JSPS KAKENHI Grant Number JP20K03579.  

\section{Preliminaries}
\label{preliminary}

\subsection{Notations and conventions}
\label{notation}

In this paper, we fix a finite dimensional Hopf algebra $H$
over a field $k$. The coproduct, the counit, and the antipode are
denoted by $\Delta$, $\varepsilon$, and $S$, respectively.
We also fix a left integral $\lambda$ in $H$.

Other notations and conventions used in this paper are summarized in the
following list.

\begin{itemize}
 \item The tensor product over $k$ is denoted by $\otimes$.

 \item The category of $k$-modules is a symmetric monoidal category
       under $\otimes$. The morphism induced by a permutation
       $\sigma$ of $\{1,2,\ldots,n\}$ is denoted by
       \[
	\sigma : M_{1}\otimes \cdots \otimes M_{n} \rarrow{}
       M_{\sigma(1)}\otimes \cdots \otimes M_{\sigma(n)}.
       \]
       For example, the symmetric monoidal structure
       $M_{1}\otimes M_{2}\to M_{2}\otimes M_{1}$ is denoted by $(1,2)$.

 \item The category of left $H$-modules is denoted by $\lMod{H}$. The
       full subcategory of finitely generated $H$-modules is denoted by
       $\lmod{H}$. 
       Given left $H$-modules $\mu_{M}:H\otimes M\to M$ and
       $\mu_{N}: H\otimes N\to N$, the left $H$-action on $M\otimes N$
       is given by 
       \[
       H\otimes M\otimes N \rarrow{\Delta\otimes 1\otimes 1} H\otimes
       H\otimes 
       M\otimes N \rarrow{(2,3)} H\otimes M\otimes H\otimes N
       \rarrow{\mu_{M}\otimes \mu_{N}} M\otimes N.
       \]
       The categories $\lMod{H}$ and $\lmod{H}$ are regarded as monoidal
       categories under this tensor product.

 \item We regard $k$ as an $H$-module via the counit $\varepsilon:H\to k$
       so that $\varepsilon$ is a morphism in $\lMod{H}$.

 \item We use Sweedler's notation for coproducts,
       i.e.
       \[
	\Delta(h) = h_{(1)}\otimes h_{(2)}
       \]
       for $h\in H$.
       We also use an analogous notation for comodules.

 \item For a category $C$ and objects $x,y$ in $C$, the set of morphisms
       from $x$ to $y$ is denoted by $C(x,y)$.
       When $C$ is small, the sets of objects and morphisms are denoted
       by $C_{0}$ and $C_{1}$, respectively. The 
       source, the target, and the unit maps are denoted by 
       \begin{align*}
	s & : C_{1} \rarrow{} C_{0} \\
	t & : C_{1} \rarrow{} C_{0} \\
	\eta & : C_{0} \rarrow{} C_{1},
       \end{align*}
       respectively.
\end{itemize}

\subsection{Module categories over a Hopf algebra and their modules}
\label{module_category}

Khovanov and Qi studied $H$-module algebras and their modules.
We would like to be slightly more general, since we are interested in
Hopfological analogues of dg categories. We regard an $H$-module algebra
as a one-object $H$-module category. 

\begin{definition}
 A \emph{left $H$-module category} is a category enriched over
 the monoidal category $\lMod{H}$. When it has a single object, it is
 called a \emph{left $H$-module algebra}. 

 By forgetting the action of $H$, we obtain the \emph{underlying
 $k$-linear category or $k$-algebra}, which is denoted by
 $\underline{A}$.
\end{definition}

We are interested in the category of left $H$-equivariant $A$-modules
for a left $H$-module category $A$. In order to give a precise
definition, we first need to fix notation and terminology for $k$-linear
categories. 
The following fact is useful to simplify notations for linear categories
and their modules.

\begin{lemma}
 \label{coproduct_decomposition_by_comodule}
 For a $k$-module $M$ and a set $S$,
 there is a one-to-one correspondence between families of submodules
 $\{M_{s}\}_{s\in S}$ indexed by $S$ with $M=\bigoplus_{s\in S} M_{s}$
 and comodule structures on $M$ over the free $k$-module
 $kS$ spanned by $S$, whose
 coalgebra structure is induced by the diagonal map on $S$. 
\end{lemma}

\begin{proof}
 Given a comodule structure $\delta:M\to kS\otimes M$,
 define 
 \[
  M_{s}=\set{m\in M}{\exists m' \text{ s.t. } \delta(m)=s\otimes m'}.
 \]
 Note that $M_{s}\cap M_{s'}=0$ if $s\neq s'$.
 Suppose $\delta(m)=\sum_{s} s\otimes m_{s}$. Then the coassociativity
 implies that $m_{s}$ belongs to $M_{s}$ and
 the counit condition implies that $\sum_{s\in S} m_{s}=m$.
 And we have $M=\bigoplus_{s} M_{s}$.

 Conversely a family of submodules with $M=\bigoplus_{s\in S} M_{s}$
 gives rise to a map $\delta: M\to kS\otimes M$ by
 $\delta(m)=s\otimes m$ for $m\in M_{s}$. This is a comodule structure
 on $M$.
\end{proof}

\begin{remark}
 This observation is due to Cohen and Montgomery
 \cite{Cohen-Montgomery1984}. The second author learned this fact from
 Hideto Asashiba.
\end{remark}

Let $A$ be a small $k$-linear category. Recall from section
\ref{introduction} that the set of objects and the modules of morphisms
in $A$ are denoted by $A_{0}$ and $A_{1}$, respectively.
By Lemma \ref{coproduct_decomposition_by_comodule}, we may regard the
total morphism space
\[
A_{1} = \bigoplus_{x,y\in A_{0}} A(x,y).
\]
as a $kA_{0}$-$kA_{0}$-bicomodule, where the right comodule structure is
given by the source map and the left comodule structure is given by the
target map. For simplicity, we use the following notations.

\begin{convention} 
 \hspace*{\fill}
\begin{itemize}
 \item The free $k$-module $kA_{0}$ is denoted by $A_{0}$ and is
       regarded as a $k$-coalgebra by the diagonal of $A_{0}$. 
 \item We identify $A$ with the module of all morphisms $A_{1}$ so that
       $A$ is an $A_{0}$-$A_{0}$-bicomodule.
\end{itemize}
\end{convention} 

Recall that given a right comodule $M$ and a left comodule $N$ over a
coalgebra $C$, the cotensor product $M\Box_{C}N$ is defined by the
equalizer
\[
 \equalizer{M\Box_{C}N}{M\otimes N}{\delta_{M}\otimes 1}{1\otimes
 \delta_{N}}{M\otimes C\otimes N},
\]
where $\delta_{M}$ and $\delta_{N}$ are comodule structure maps for $M$
and $N$, respectively.
With this notation, we may identify
\[
 A\Box_{A_{0}} A = \left(\bigoplus_{y,z\in A_{0}}
 A(y,z)\right)\Box_{A_{0}}\left(\bigoplus_{x,y\in A_{0}} A(x,y)\right) 
 \cong \bigoplus_{x,y,z\in A_{0}} A(y,z)\otimes A(x,y)
\]
so that the composition of morphisms and the unit are
given by bicomodule maps 
\begin{align*}
 \mu_{A} & : A\Box_{A_{0}} A \rarrow{} A \\
 \eta_{A} & : A_{0} \rarrow{} A. 
\end{align*}
In other words, we regard a $k$-linear category as a monoid objects in
the monoidal category of $A_{0}$-$A_{0}$-bicomodules whose monoidal
structure is given by $\Box_{A_{0}}$.

\begin{definition}
 \label{def:module_over_category}
 Let $A$ be a $k$-linear category. A \emph{left $A$-module} consists of
 \begin{itemize}
  \item a left $A_{0}$-comodule $M$, and
  \item a morphism of left $A_{0}$-comodules
	$\mu_{A,M}: A\Box_{A_{0}}M\to M$,
 \end{itemize}
 which satisfy the unit and the associativity conditions.
 For left $A$-modules $M$ and $N$, a morphism of left $A_{0}$-comodules
 $f:M\to N$ is called an \emph{$A$-module homomorphism} if it commutes with
 the actions of $A$.  

 The category of left $A$-modules and $A$-module homomorphisms is
 denoted by $\lMod{A}$.
\end{definition}

\begin{remark}
 Thanks to Lemma \ref{coproduct_decomposition_by_comodule}, a left
 $A$-module $M$ can be regarded as a collection $\{M(x)\}$ of
 $k$-modules indexed by objects of $A$ equipped with a family of
 $k$-linear maps
 \[
  A(x,y)\otimes M(x) \rarrow{} M(y)
 \]
 satisfying the associativity and the unit conditions.
 In other words, a left $A$-module is nothing but a functor
 $A\to\lMod{k}$. Similarly, a right $A$-module is a contravariant
 functor from $A$ to $\lMod{k}$.
\end{remark}

When $A$ is a left $H$-module category, we need to incorporate the
action of $H$ as follows.

\begin{definition}
 \label{def:equivariant_module_over_module_algebra}
 Let $A$ be a left $H$-module category. A \emph{left $H$-equivariant
 $A$-module} consists of
 \begin{itemize}
  \item a left $A$-module $\mu_{A,M}: A\Box_{A_{0}}M\to M$ and
  \item a left $H$-module structure
	$\mu_{H,M}: H\otimes M\to M$ on $M$
 \end{itemize}
 satisfying the following conditions:
 \begin{enumerate}
  \item $\mu_{H,M}$ is a morphism of $A_{0}$-comodules.
  \item $\mu_{A,M}$ and $\mu_{H,M}$ are compatible in the sense that the
	following diagram is commutative
	\[
	\xymatrix{
	 H\otimes A\Box_{A_{0}}M 
	  \ar[d]_{\Delta\otimes 1\Box 1}
	  \ar[rrr]^{1\otimes\mu_{A,M}} & & &
	  H\otimes M \ar[dd]^{\mu_{H,M}} \\ 
	  H\otimes H \otimes A\Box_{A_{0}}M \ar[d]_{(2,3)} \\
	  (H\otimes A)\Box_{A_{0}}(H\otimes M)
	  \ar[rr]_(.6){\mu_{H,A}\otimes \mu_{H,M}}
	  & & A\Box_{A_{0}}M \ar[r]_(.55){\mu_{A,M}} & M.
	}
	\]
 \end{enumerate}

 Given two $H$-equivariant $A$-modules $M$ and $N$, an
 \emph{$H$-equivariant morphism} from $M$ to $N$ is a morphism
 $f:M\to N$ of left $A_{0}$-comodules which commutes with both
 $A$-module structures and $H$-module structures.

 The category of $H$-equivariant left $A$-modules and
 $\underline{A}$-module homomorphisms is denoted by $\bm{C}_{A,H}$.
 The wide subcategory of $H$-equivariant morphisms in
 $\bm{C}_{A,H}$ is denoted by $\bm{C}_{A,H}^{H}$.
\end{definition}

\begin{remark}
 When $A=k$ with the $H$-action given by $\varepsilon:H\to k$, we have
 an identification $\bm{C}_{k,H}^{H}=\lMod{H}$.
\end{remark}

The following fact plays a fundamental role in Hopfological algebra.
See section 5.1 of Qi's paper \cite{1205.1814} for the case of
$H$-module algebra. It is straightforward to obtain a generalization to
the case of $H$-module category.

\begin{proposition}
 \label{module_category_of_modules_over_module_algebra}
 For a left $H$-module category $A$ and left $H$-equivariant $A$-modules
 $M,N$, 
 define a left $H$-action on $\bm{C}_{A,H}(M,N)$ by
 \[
 (hf)(m) = h_{(2)}f\left(S^{-1}(h_{(1)})m\right)
 \]
 for $h\in H$, $f\in \bm{C}_{A,H}(M,N)=(\lMod{A})(M,N)$, and $m\in M$,
 where $S$ is the antipode of $H$.
 Then $\bm{C}_{A,H}(M,N)$ becomes a left $H$-module with which the
 compositions of morphisms are $H$-module homomorphisms and the identity
 morphisms are $H$-invariant.  In other words, $\bm{C}_{A,H}$ becomes a
 left $H$-module category.
\end{proposition}

Recall that for an $H$-module $V$, the \emph{submodule of
invariants} is defined by
\[
 V^{H} = \set{x\in V}{hx=\varepsilon(h)x \text{ for all }
 h\in H}. 
\]
Our notation $\bm{C}_{A,H}^{H}$ is designed to fit into the
following identification.

\begin{corollary}
 \label{invariant_and_equivariant}
 Let $f$ be a morphism in $\bm{C}_{A,H}$.
 Then $hf=\varepsilon(h)f$ for all
 $h\in H$ if and only if $f$ is an $H$-module homomorphism.
 In other words, under the $H$-module structure on $\bm{C}_{A,H}(M,N)$
 in Proposition \ref{module_category_of_modules_over_module_algebra}, we
 have
 \[
 \left(\bm{C}_{A,H}(M,N)\right)^{H} =
 \bm{C}_{A,H}^{H}(M,N).  
 \]
\end{corollary}

A left $H$-module algebra $A$ gives rise to a new algebra $A\# H$ by the
smash product construction. The construction can be extended to
$H$-module categories.
 
\begin{definition}
 Let $A$ be a left $H$-module category. Define a $k$-linear
 category $A\# H$ as follows.
 The objects are
 \[
  (A\# H)_{0} = A_{0}.
 \]
 For $x,y\in (A\# H)_{0}$, the module of morphisms from $x$ to $y$ is
 \[
  (A\# H)(x,y) = A(x,y)\otimes H
 \]
 so that the module of morphisms is 
 \[
  A\# H = \bigoplus_{x,y\in A_{0}} A(x,y)\otimes H= A\otimes H. 
 \]
 The unit is given by
 \[
 A_{0} \cong A_{0}\otimes k \rarrow{\eta_{A}\otimes \eta_{H}} A\otimes
 H,
 \]
 where $\eta_{H}:k\to H$ is the unit of $H$.
 The composition of morphisms is given by 
 \begin{eqnarray*}
  (A\otimes H)\Box_{A_{0}} (A\otimes H)
  & \rarrow{1\otimes\Delta\otimes 1\otimes 1} &
  (A\otimes(H\otimes H)) \Box_{A_{0}} (A\otimes H) \\
  & \rarrow{(3,4)} & A\Box_{A_{0}} (H\otimes A) \otimes
  (H\otimes H) \\
  & \rarrow{1\otimes\mu_{H,A}\otimes\mu_{H}} & A\Box_{A_{0}} A\otimes
  H \\
  & \rarrow{\mu_{A}\otimes 1} & A\otimes H.
 \end{eqnarray*}
\end{definition}

The following description of $\bm{C}_{A,H}^{H}$ is well known when $A$
is an $H$-module algebra. 
The proof is essentially the same as the case of $H$-module algebras and
is omitted.

\begin{proposition}
 \label{equivariant_module_as_module_over_smash_product}
 For an $H$-module category $A$, the category $\bm{C}_{A,H}^{H}$ is
 equivalent to $\lMod{(A\#H)}$. 
\end{proposition}

Another important fact in Hopfological algebra is that
$\bm{C}_{A,H}^{H}$ is a module category\footnote{Should not be
confused with an $H$-module category, which means a category enriched over
$\lMod{H}$.} over the monoidal category 
$(\lMod{H},\otimes,k)$. 
Following Qi's paper \cite{1205.1814}, we use a right action.
Recall that a right action of a monoidal category $(\bm{V},\otimes,1)$
on a category $\bm{C}$ consists of a functor
\[
 \otimes : \bm{C}\times\bm{V} \rarrow{} \bm{C},
\]
a natural isomorphism
\[
a: (X\otimes V)\otimes W \rarrow{\cong} X\otimes (V\otimes W)
\]
for $X \in\bm{C}_{0}$ and $V,W\in\bm{V}_{0}$ satisfying the pentagon
axiom, and a natural isomorphism 
\[
r : X\otimes 1 \rarrow{\cong} X
\]
for $X\in\bm{C}_{0}$ satisfying the unit axiom.
A precise definition can be found, for example, in Ostrik's paper
\cite{math/0111139}, who refers to Bernstein's lecture note
\cite{q-alg/9501032} and a paper \cite{Crane-Frenkel1994} by Crane and
Frenkel for the first appearance in the literature. 

In the case of $\bm{C}_{A,H}^{H}$, the right action of $\lMod{H}$ is
given by $M\otimes V$ for $M$ in $\bm{C}_{A,H}^{H}$ and $V$ in
$\lMod{H}$. The left $A$-module structure is given by that of $M$ and
the left $H$-module structure is given by the composition
\[
 H\otimes M\otimes V \rarrow{\Delta\otimes 1\otimes 1}  
 H\otimes H\otimes M\otimes V \rarrow{(2,3)} H\otimes M\otimes H\otimes
 V \rarrow{} M\otimes V,
\]
where the last map is given by the $H$-module structures of $M$ and $V$.

This action of $\lMod{H}$ allowed Khovanov to introduce functors
\[
 C,\Sigma:\bm{C}_{A,H}^{H} \rarrow{} \bm{C}_{A,H}^{H}
\]
by
\begin{align*}
 C(M) & = M\otimes H \\
 \Sigma(M) & = M\otimes (H/(\lambda)),
\end{align*}
respectively.
These are called the \emph{cone} and the \emph{suspension} functors,
respectively. 

The cone functor allows us to define an analogue of chain homotopy.

\begin{definition}
 Two morphisms $f,g:M\to N$ in $\bm{C}_{A,H}^{H}$ are called
 \emph{homotopic} if there exists a morphism $\varphi: C(M) \to N$
 making the diagram 
 \[
 \xymatrix{
   M \ar[d]_{\cong} \ar[rr]^{f-g} &
   & N \\
   M\otimes k \ar[r]_(.45){1\otimes\lambda} & M\otimes H
   \ar@{=}[r] & C(M) \ar[u]_{\varphi} 
 }
 \]
 commutative, in which case, we denote $f\simeq g$. The morphism
 $\varphi$ is called a \emph{homotopy} from $f$ to $g$.

 The \emph{homotopy category} or the \emph{stable category} of
 $\bm{C}_{A,H}^{H}$, denoted by $\cT_{A,H}^{H}$, is the category having
 the same objects as $\bm{C}_{A,H}^{H}$ whose set of morphisms from 
 $M$ to $N$ is given by
 \[
  \cT_{A,H}^{H}(M,N) = \bm{C}_{A,H}^{H}(M,N)/_{\simeq}. 
 \]
 When $A=k$, under the identification $\bm{C}_{k,H}^{H}=\lMod{H}$, we
 denote
 \[
  \ho(\lMod{H}) = \cT_{k,H}^{H}.
 \]
\end{definition}

Khovanov noticed that the homotopy category $\cT_{A,H}^{H}$ has a
structure of triangulated category with $\Sigma$ a shift functor.
An inverse to $\Sigma$ on $\cT_{A,H}^{H}$ is given by
\[
 \Sigma^{-1}(M) = M\otimes (\Ker\varepsilon),
\]
which is called the \emph{desuspension} functor. 
 
\begin{lemma}
 \label{Sigma}
 The functors $\Sigma$ and $\Sigma^{-1}$ induce functors on
 $\cT_{A,H}^{H}$ that are inverse to each other.
\end{lemma}

\begin{proof}
 Khovanov showed that there exists a projective $H$-module $Q$ such that 
 \[
  \Ker(\varepsilon)\otimes (H/(\lambda)) \cong k\oplus Q.
 \]
 Thus 
 \[
  \Sigma(\Sigma^{-1}(M)) \cong M\oplus M\otimes Q
 \]
 for any $M$ in $\bm{C}_{A,H}^{H}$.
 By Lemma \ref{tensoring_with_H} below, we have
 $M\otimes Q\cong 0$ in $\cT_{A,H}^{H}$.

 Similarly we see that $\Sigma^{-1}(\Sigma(M))\cong M$ in
 $\cT_{A,H}^{H}$. 
\end{proof}

The following fact used in the above proof is stated as Proposition 2 in
Khovanov's paper \cite{math/0509083}. Khovanov refers to Montgomery's
book \cite{Montgomery1993} for a proof.

\begin{lemma}
 \label{tensoring_with_H}
 For any $H$-module $M$, $H\otimes M$ is a free $H$-module.
 If $M$ is of finite dimensional over $k$. Then
 $H\otimes M$ is a free $H$-module of rank $\dim_{k}M$. 
\end{lemma}

\begin{corollary}
 Let $P$ be a projective $H$-module. Then, for any $H$-module $M$,
 both $P\otimes M$ and $M\otimes P$ are projective $H$-modules.
\end{corollary}

In order to describe distinguished triangles, let us recall the
definition of mapping cones from Khovanov's paper, which is essentially
the same as the definition of mapping cones in the category of chain
complexes. 

\begin{definition}
 Given a morphism $f:M\to N$ in $\bm{C}_{A,H}^{H}$, the pushout
 of the extension
 \[
 0 \rarrow{} M \rarrow{i_{M}} C(M) \rarrow{p_{M}} \Sigma(M) \rarrow{} 0
 \]
 along $f$ is denoted by
 \[
 0 \rarrow{} N \rarrow{j_{f}} C_{f} \rarrow{\delta_{f}} \Sigma(M)
 \rarrow{} 0.
 \]
 Here $i_{M}$ is the composition
 $M\cong M\otimes k\rarrow{1\otimes\lambda} M\otimes H=C(M)$.
 The object $C_{f}$ is called the \emph{mapping cone} of $f$.

 A sequence of the form
 \[
 M \rarrow{f} N \rarrow{j_{f}} C_{f} \rarrow{\delta_{f}} \Sigma(M)
 \]
 for a morphism $f:M\to N$ is called a \emph{standard triangle} in
 $\cT_{A,H}^{H}$. 
\end{definition}

\begin{theorem}[Khovanov]
 The category $\cT_{A,H}^{H}$ becomes a triangulated category 
 with $\Sigma$ a shift functor,
 by declaring a sequence $X\to Y\to Z$ in $\cT_{A,H}^{H}$ to be a
 distinguished triangle if it is isomorphic to a standard triangle. 
\end{theorem}

The following fact can be verified immediately.

\begin{lemma}
 \label{suspension_and_Hom}
 Let $M$ and $N$ be $H$-equivariant $A$-modules.
 For any $H$-module $V$, the canonical isomorphism of $k$-modules
 \[
  \bm{C}_{A,H}(M,N)\otimes V \rarrow{} \bm{C}_{A,H}(M,N\otimes V)
 \]
 is an isomorphism of $H$-modules.
 In particular, we have isomorphisms of $H$-modules
 \begin{align*}
  \Sigma\bm{C}_{A,H}(M,N) & \cong \bm{C}_{A,H}(M,\Sigma N) \\
  \Sigma^{-1}\bm{C}_{A,H}(M,N) & \cong
  \bm{C}_{A,H}(M,\Sigma^{-1} N).
 \end{align*}
\end{lemma}

Recall that the underlying $k$-linear category of an $H$-module category
$A$ is denoted by $\underline{A}$.

\begin{definition}
 Let $U : \bm{C}_{A,H}^{H} \to \lMod{\underline{A}}$ be the forgetful
 functor.  
 We say a short exact sequence
 \[
 0 \rarrow{} L \rarrow{f} M \rarrow{g} N\rarrow{} 0
 \]
 $\bm{C}_{A,H}^{H}$ is \emph{$A$-split} if 
 $0\to U(L)\to U(M)\to U(N)\to 0$ is a split short exact sequence in
 $\lMod{\underline{A}}$. 
\end{definition}

Qi found a characterization of distinguished triangles in
$\cT_{A,H}^{H}$ in terms of $A$-split sequences.
See Lemma 4.3 of \cite{1205.1814}.  

\begin{lemma}
 \label{Hopf_mapping_cone}
 Let
 \[
  0 \rarrow{} L \rarrow{f} M \rarrow{g} N\rarrow{} 0
 \]
 be an $A$-split short exact sequence in $\bm{C}_{A,H}^{H}$.
 Then there exists a distinguished triangle
 in $\cT_{A,H}^{H}$ of the form
 \[
  L \rarrow{[f]} M \rarrow{[g]} N \rarrow{\delta} \Sigma(L). 
 \]
 Conversely, any distinguished triangle in
 $\cT_{A,H}^{H}$ is isomorphic to the one that arises 
 from an $A$-split short exact sequence in $\bm{C}_{A,H}^{H}$.
\end{lemma}

The following useful fact is proved as Lemma 4.4 in Qi's paper
and used in the proof of Lemma \ref{Hopf_mapping_cone}.

\begin{lemma}
 \label{cone_is_perp_to_everything}
 Any $A$-split extension of the form
 \[
  0 \rarrow{} L \rarrow{} M \rarrow{} Z\otimes H \rarrow{} 0
 \]
 splits.
\end{lemma}

\begin{definition}
 We say a short exact sequence
 \[
 0 \rarrow{} L \rarrow{f} M \rarrow{g} N \rarrow{} 0
 \]
 in $\bm{C}_{A,H}^{H}$ \emph{homotopically splits} if there
 exists a morphism $s: N\to M$ with $g\circ s\simeq 1_{N}$.
 The morphism $s$ is called a \emph{homotopy section} of $g$.
\end{definition}

\begin{corollary}
 Let
 \[
 0 \rarrow{} L \rarrow{f} M \rarrow{g} N \rarrow{} 0
 \]
 be an $A$-split short exact sequence in $\bm{C}_{A,H}^{H}$.
 Then it homotopically splits if and only if
 there exists a morphism $t: M\to L$ with $f\circ t\simeq 1_{L}$.
\end{corollary}

\begin{proof}
 Since the sequence splits in $\lMod{\underline{A}}$, it defines a
 triangle 
 \[
  L \rarrow{[f]} M \rarrow{[g]} N \rarrow{\delta} \Sigma(L)
 \]
 in $\cT_{A,H}^{H}$.
 Then $f$ has a homotopy section if and only if this triangle is
 isomorphic to the trivial triangle, which in turn is equivalent to
 saying that $[f]$ is a section, or there exists a morphism $t: M\to L$
 with $f\circ t\simeq 1_{L}$.
\end{proof}

We have the following closely related fact.

\begin{lemma}
 \label{mapping_cone_of_null_homotopic_map}
 A morphism $f:X\to Y$ in $\bm{C}_{A,H}^{H}$ is homotopic to $0$, if and
 only if there exists an isomorphism of extensions
 \[
 \xymatrix{
 0 \ar[r] & Y \ar@{=}[d] \ar[r]^{j_{f}}
   & C_{f} \ar[d]^{\cong}
   \ar[r]^{\delta_{f}} &  \Sigma(X) \ar@{=}[d] \ar[r]
   & 0 \\   
 0 \ar[r] & Y \ar[r]^(.35){(1_{Y},0)} & Y\oplus \Sigma(X)
 \ar[r]^(.55){\pr_{2}} & \Sigma(X) \ar[r] & 0.
 }
 \]
\end{lemma}

\begin{proof}
 Suppose $f\simeq 0$.
 By definition, there exists $h:C(X)\to Y$ such
 that $f=h\circ i_{X}$, which gives rise to a morphism $r:C_{f}\to Y$
 making the diagram
 \[
 \xymatrix{
  X \ar[d]_{i_{X}} \ar[r]^{f} & Y \ar[d]_{j_{f}} \ar[ddr]^{1_{Y}} 
   & \\ 
 C(X) \ar[drr]_{h} \ar[r]^{q_{f}} & C_{f} \ar@{.>}[dr]^(.3){r} 
   & \\ 
   & & Y
 }
 \]
 commutative.
 Then we have
 \[
  (1-j_{f}\circ r)\circ j_{f} = j_{f} - j_{f} = 0.
 \]
 Since $\delta_{f}$ is a cokernel of $j_{f}$, and we obtain a
 morphism $s: \Sigma(X)\to C_{f}$ with 
 \[
  1-j_{f}\circ r = s\circ\delta_{f}.
 \]
 Then the maps
 \begin{align*}
  (r,\delta_{f}) & : C_{f} \rarrow{} Y\oplus\Sigma(X) \\
  j_{f}+s & : Y\oplus\Sigma(X) \rarrow{} C_{f}
 \end{align*}
 are inverse to each other and we obtain an isomorphism of extensions
 that we wanted.

 Conversely, suppose we have a map
 $\varphi:C_{f}\to Y\oplus\Sigma(X)$ which defines an isomorphism of
 extensions. Then in the pushout diagram
 \[
 \xymatrix{
  X \ar[r]^{f} \ar[d]_{i_{X}}
   & Y \ar[d]^{j_{f}} \\ 
 C(X) \ar[r]_{q_{f}} & C_{f},
 }
 \]
 the composition $\pr_{2}\circ\varphi:C_{f}\to Y$ defines a left inverse
 to $j_{f}$ and thus $f\simeq 0$.
\end{proof}

The mapping cone construction has the following nice property.

\begin{lemma}
 \label{mapping_cone_and_A-hom}
 Let $f: M\to N$ be a morphism in $\bm{C}_{A,H}^{H}$ and $P$ be an
 object of $\bm{C}_{A,H}^{H}$ which is projective as an
 $\underline{A}$-module. 
 Let $\bm{C}_{A,H}(P,f): \bm{C}_{A,H}(P,M)\to \bm{C}_{A,H}(P,N)$ be the
 morphism induced by $f$, then we have a natural isomorphism 
 \[
  C_{\bm{C}_{A,H}(P,f)} \cong \bm{C}_{A,H}(P,C_{f}).
 \]
\end{lemma}

\begin{proof}
 Since $P$ is projective as an $\underline{A}$-module,
 $\bm{C}_{A,H}(P,-)$ is an 
 exact functor and we obtain a diagram of short exact sequences
 \[
 \xymatrix{
  0 \ar[r] & \bm{C}_{A,H}(P,N) \ar@{=}[d] \ar[r]
   & C_{\bm{C}_{A,H}(P,f)} \ar[d] \ar[r]
   & \Sigma\bm{C}_{A,H}(P,M) \ar[d]^{\cong}
   \ar[r] & 0 \\
   0 \ar[r] & \bm{C}_{A,H}(P,N) \ar[r]
   & \bm{C}_{A,H}(P,C_{f}) 
   \ar[r] & \bm{C}_{A,H}(P,\Sigma(M)) 
   \ar[r] & 0,
 }
 \]
 where the middle vertical arrow is the morphism obtained by the
 universality of pushout. Since the right vertical arrow is an
 isomorphism by Lemma \ref{suspension_and_Hom}, so is the middle
 vertical arrow.
\end{proof}

\subsection{Homological algebra in Hopfological algebra}
\label{homological_algebra_in_stable_module_category}

In order to perform homological algebra in $\bm{C}_{A,H}^{H}$,
$\bm{C}_{A,H}$, and $\cT_{A,H}^{H}$, we need homology.
Let us recall the definition from Qi's paper.

\begin{definition}
 For a left $H$-module $M$, we denote
 \begin{align*}
  Z(M) & = M^{H} =
  \set{m\in M}{hm=\varepsilon(h)m \text{ for all $h\in H$}} \\ 
  B(M) & = \lambda M \\
  H(M) & = Z(M)/B(M).  
 \end{align*}
 The functor $H: \lMod{H} \to \lMod{k}$
 is called the \emph{canonical homological functor}. 
 The composition with the forgetful functor
 $\bm{C}_{A,H}^{H}\to\lMod{H}$ is also denoted by
 \[
  H : \bm{C}_{A,H}^{H} \rarrow{} \lMod{k}.
 \]
\end{definition}

\begin{example}
 \label{projective_module_is_acyclic}
 Suppose $P$ is projective as an $H$-module.
 There exists a free $H$-module $F$ with $F\cong P\oplus Q$ as 
 $H$-modules. 
 The homology of $H$ is trivial, since $Z(H)$ is the submodule of
 integrals, which is known to be of $1$-dimensional over $k$ generated
 by a fixed integral $\lambda$.
 It implies that $H(F)=0$, and we have $H(P)=0$. 
 In particular, $H(C(M))=0$ for any $H$-equivariant $A$-module $M$ and
 we see that the canonical homological functor descends to 
 \[
 H: \cT_{A,H}^{H} \rarrow{} \lMod{k}.
 \]
\end{example}

\begin{example}
 \label{H-homotopy_class_by_homology} 
 Let $M$ and $N$ be $H$-equivariant $A$-modules. Then by Corollary
 \ref{invariant_and_equivariant}, we have 
 \[
  Z\left(\bm{C}_{A,H}(M,N)\right) = 
 \bm{C}_{A,H}^{H}(M,N). 
 \]
 For morphisms $f,g:M\to N$ in
 $\bm{C}_{A,H}^{H}$, 
 $f \simeq g$ if and only if $f-g\in
 B\left(\bm{C}_{A,H}(M,N)\right)$
 by definition.

 Thus $H(\bm{C}_{A,H}(M,N))$ can be
 identified with the set of $H$-homotopy classes of
 $H$-equivariant $A$-module maps from $M$ to $N$. 
 In other words,
 \[
  H\left(\bm{C}_{A,H}(M,N)\right) =
 \cT_{A,H}^{H}(M,N). 
 \]
 In particular, we have an isomorphism
 \[
 \Ext^{1}_{\cT_{A,H}^{H}}(M,N) \cong \cT_{A,H}^{H}(M,\Sigma(N)) =
 H(\bm{C}_{A,H}(M,\Sigma(N))). 
 \]
\end{example}

\begin{proposition}
 \label{H-homology_is_homological}
 The canonical homological functor is homological, i.e.~any triangle in
 $\cT_{A,H}^{H}$ induces a long exact sequence by the canonical
 homological functor $H$.
\end{proposition}

\begin{proof}
 For a left $H$-module $M$, we have an identification
 \[
  M \cong (\lMod{H})(k,M) = \bm{C}_{k,H}^{H}(k,M).
 \]
 in $\lMod{H}=\bm{C}_{k,H}^H$ and we have
 \[
 H(M) \cong \cT_{k,H}^{H}(k,M) 
 \]
 by the previous example. 
 Since $\cT_{k,H}^{H}$ is a triangulated category, this is a
 homological functor.
 The functor $\cT_{A,H}^{H}\to \cT_{k,H}^{H}$ which forgets $A$-module
 structures preserves triangles and thus
 \[
  H : \cT_{A,H}^{H}\rarrow{} \lMod{k}
 \]
 is also homological.
\end{proof}

\begin{definition}
 \label{def:quism_for_H-module}
 A morphism $f:M\to N$ in $\bm{C}_{A,H}^{H}$ is
 called a \emph{quasi-isomorphism} or a \emph{quism} if
 the induced map $H(f): H(M)\to H(N)$ is an isomorphism in
 $\lMod{k}$.
 It is called a \emph{$\Sigma$-quism} if
 $H(\Sigma^{n}(f)): H(\Sigma^{n}(M))\to H(\Sigma^{n}(N))$ is an
 isomorphism for all $n\in\Z$.
 The wide subcategory of $\Sigma$-quisms is denoted by
 $\Quism^{\Sigma}$. 

 An object $M$ of $\bm{C}_{A,H}^{H}$ is called
 \emph{acyclic} if $H(M)=0$.
 The class of acyclic objects is denoted by
 $\Triv_{A,H}$.
 If $H(\Sigma^{n}M)=0$ for all $n\in \Z$, $M$ is called
 \emph{$\Sigma$-acyclic}. The class of $\Sigma$-acyclic
 objects is denoted by $\Triv_{A,H}^{\Sigma}$.
 We regard them as full subcategories of $\bm{C}_{A,H}^{H}$ or
 $\cT_{A,H}^{H}$. 
\end{definition}


\begin{remark}
 Khovanov \cite{math/0509083} and Qi \cite{1205.1814} used a different
 notion of quasi-isomorphisms. A morphism $f: M\to N$ in
 $\bm{C}_{A,H}^{H}$ is a quasi-isomorphism in their sense if it is a
 homotopy equivalence in $\lMod{H}$.
 Hence their acyclic objects are different from ours.
\end{remark}

\begin{lemma}
 \label{Triv*_is_thick}
 The category $\Triv_{A,H}^{\Sigma}$ is a thick subcategory of
 both $\bm{C}_{A,H}^{H}$ and $\cT_{A,H}^{H}$. 
\end{lemma}

\begin{proof}
 Since the canonical homological functor commutes with direct sums,
 $\Triv_{A,H}^{\Sigma}$ is closed under taking direct summands.
 The two-out-of-three property follows from the fact that $H$
 is a homological functor and the fact that
 $\Triv_{A,H}^{\Sigma}$ is closed under $\Sigma$. 
\end{proof}

\begin{example}
 \label{suspension_of_cone_is_acyclic}
 For any $M$ and $n\in\Z$, $\Sigma^{n}C(M)$ is acyclic,
 i.e.~$C(M)$ is $\Sigma$-acyclic.  
 This can be verified as follows.

 Suppose $n\ge 0$.
 By Lemma \ref{switching_lambda}, we have an isomorphism of $H$-modules
 \[
 \Sigma^{n}C(M) = M\otimes H\otimes (H/(\lambda))^{\otimes n} \cong
 H \otimes M\otimes (H/(\lambda))^{\otimes n}.
 \]
 This is a free left $H$-module by Lemma \ref{tensoring_with_H} and
 hence is acyclic by Example \ref{projective_module_is_acyclic}. 
 By replacing $H/(\lambda)$ by $\Ker\varepsilon$, we see that
 $\Sigma^{n}C(M)$ is acyclic for $n<0$. 
\end{example} 

The following fact, used in the above argument, appears as Lemma 2 in
Khovanov's paper \cite{math/0509083}. 

\begin{lemma}
 \label{switching_lambda}
 In the category of $H$-modules, there exists an isomorphism
 $r: V\otimes H\to H\otimes V$ which is natural in $V$ and
 makes the following diagram commutative.
 \[
 \xymatrix{
  V\otimes k \ar[d]_{1_{V}\otimes\lambda} & V
   \ar[l]^(.4){\cong} \ar[r]_(.4){\cong} 
   & k\otimes V \ar[d]^{\lambda\otimes 1_{V}}  \\
   V\otimes H \ar[rr]_{r} & & H\otimes V.
 }
 \]
\end{lemma}

The following is an analogue of Proposition 2.3.5 (1) in
\cite{Avramov-Foxby-Halperin}.

\begin{lemma}
 \label{surjectivity}
 Let $M, N$ be objects of $\bm{C}_{A,H}^{H}$.
 If $f:M\to N$ is a surjective quism, then both $B(f): B(M)\to B(N)$ and
 $Z(f): Z(M)\to Z(N)$ are surjective.
\end{lemma}

\begin{proof}
 The morphism $f$ gives rise to a commutative diagram
 \[
 \xymatrix{
 M \ar[d]_{f} \ar[r] & B(M) \ar[d]^{B(f)}  \\
 N \ar[r] & B(N).
 }
 \]
 Since horizontal arrows are surjective,
 if $f$ is surjective, then so is $B(f)$.
 
 By assumption, $H(f)$ is an isomorphism.
 The commutativity of the diagram of extensions
 \[
 \xymatrix{
 0 \ar[r] & B(M) \ar[d]_{B(f)} \ar[r] & Z(M)
 \ar[d]^{Z(f)} \ar[r] & H(M) \ar[d]^{H(f)} \ar[r] & 0 \\
 0 \ar[r] & B(N) \ar[r] & Z(N) \ar[r] & H(N) \ar[r] & 0
 }
 \]
 implies that $Z(f)$ is surjective.
\end{proof}

Recall that the forgetful functor from $\bm{C}_{A,H}^{H}$ to
$\lMod{\underline{A}}$ is denoted by $U$.
Left and right adjoints to this functor are 
useful in studying the orthogonality in $\bm{C}_{A,H}^{H}$.
By regarding an $\underline{A}$-module $M$ as a trivial $H$-module, the 
cone functor $C:\bm{C}_{A,H}^{H}\to \bm{C}_{A,H}^{H}$ gives us a
functor 
\[
 C : \lMod{\underline{A}} \rarrow{} \bm{C}_{A,H}^{H}.
\]

The following fact is obvious.

\begin{lemma}
 The cone functor $C$ is an exact functor which is left adjoint to $U$. 
\end{lemma}

The functor $U$ also has a right adjoint.

\begin{definition}
 Define a functor
 \[
  E : \lMod{\underline{A}} \rarrow{} \lMod{\underline{A}}
 \]
 by $E(M) = (\lMod{k})(H, M)$. The $\underline{A}$-module structure is
 defined by 
 $(a\varphi)(g) = a\varphi(g)$ 
 for $a\in A$, $\varphi\in E(M)$, and $g\in H$.
\end{definition}

\begin{lemma}
 For $\varphi\in E(M)$, define an action of $h\in H$ on
 $\varphi$ by 
 \[
  (h\cdot \varphi)(g) = \varphi(S(h)g)
 \]
 Then it defines a left $H$-module structure on $E(M)$.
 It is compatible with the action of $A$ and we obtain an exact functor
 \[
  E : \lMod{\underline{A}} \rarrow{} \bm{C}_{A,H}^{H}.
 \]
\end{lemma}

\begin{proof}
 Let us verify the associativity. For $h,h',g\in H$,
 \begin{align*}
  (h\cdot (h'\cdot \varphi))(g) & = (h\cdot
  \varphi)(S(h')g) \\
  & = \varphi(S(h')S(h)g) \\
  & = \varphi(S(hh')g) \\
  & = ((hh')\cdot \varphi)(g).
 \end{align*}
 We also have $1\cdot \varphi=\varphi$, since $S(1)=1$.

 In order to verify that $E(M)$ is an $H$-equivariant $A$-module, let
 $a\in A$, $\varphi\in E(M)$, and $g,h\in H$.
 Then
 \begin{align*}
  \left((h_{(1)}a)(h_{(2)}\varphi)\right)(g) & = (h_{(1)}a)
  ((h_{(2)}\varphi)(g)) \\
  & = (h_{(1)}a) \varphi(S(h_{(2)})g) \\
  & = (\varepsilon(h_{(1)})a) \varphi(S(h_{(2)})g) \\
  & = a \varphi(S(\varepsilon(h_{(1)})h_{(2)})g) \\
  & = a \varphi(S(h)g) \\
  & = (h\cdot (a\varphi))(g), 
 \end{align*}
 which means that the $A$-module structure on $E(M)$ is compatible with
 the $H$-module structure.

 Since $k$ is a field, $E$ is an exact functor.
\end{proof}

It is a fundamental fact that if $H$ is finite dimensional, the
antipode $S$ is bijective.
See \cite{Larson-Sweedler1969} or
Corollary 5.1.6 of Sweedler's book \cite{SweedlerHopfAlgebras}, for
example. 

\begin{proposition}
 \label{P_is_right_adjoint_to_U}
 The functor $E$ is right adjoint to the forgetful functor 
 $U:\bm{C}_{A,H}^{H}\to\lMod{\underline{A}}$.
\end{proposition}

\begin{proof}
 For an $H$-equivariant $A$-module $M$ and an $A$-module $N$, define
 \[
 \Phi : (\lMod{\underline{A}})\left(U(M),N\right) \rarrow{}
 \bm{C}_{A,H}^{H}\left(M,E(N)\right)  
 \] 
 by
 \[
  \Phi(\varphi)(m)(h) = \varphi\left(S^{-1}(h)m\right)
 \]
 for $\varphi\in (\lMod{\underline{A}})\left(U(M),N\right)$, $m\in M$ and
 $h\in H$. 
 Then $\Phi(\varphi)$ is a $H$-module homomorphism, since
 \begin{align*}
  \Phi(\varphi)(h'm)(h) & =
  \varphi\left(S^{-1}(h)h'm\right) \\ 
 \left(h'\cdot \left(\Phi(\varphi)(m)\right)\right)(h) & =
  \left(\Phi(\varphi)(m)\right)(S(h')h) \\
  & = \varphi\left(S^{-1}(S(h')h)m\right) \\
  & = \varphi\left(S^{-1}(h)S^{-1}(S(h'))m\right) \\
  & = \varphi\left(S^{-1}(h)h' m\right).
 \end{align*}
 Define
 \[
  \Psi :  \bm{C}_{A,H}^{H}\left(M,E(N)\right)
 \rarrow{} (\lMod{\underline{A}})\left(U(M),N\right)
 \]
 by
 \[
  \Psi(\psi)(m) = \psi(m)(1).
 \]
 Then $\Phi$ and $\Psi$ are inverse to each other, since
 \begin{align*}
  \left((\Psi\circ\Phi)(\varphi)\right)(m) & = \Psi(\Phi(\varphi))(m) \\
  & = \Phi(\varphi(m))(1) \\
  & = \varphi(S^{-1}(1)m) \\
  & = \varphi(m) \\
  \left(\left((\Phi\circ\Psi)(\psi)\right)(m)\right)(h)
  & = \left(\left(\Phi(\Psi(\psi))\right)(m)\right)(h) \\
  & = \Psi(\psi)(S^{-1}(h)m) \\
  & = \psi(S^{-1}(h)m)(1) \\
  & = \left(S^{-1}(h)\cdot \psi(m)\right)(1) \\
  & = \psi(m)(S(S^{-1}(h))1) \\
  & = \psi(m)(h).
 \end{align*}
\end{proof}

\begin{corollary}
 \label{U_of_projective_is_projective}
 Under the assumption of Proposition \ref{P_is_right_adjoint_to_U}, if
 $P$ is projective in $\bm{C}_{A,H}^{H}$, then
 $U(P)$ is projective in $\lMod{\underline{A}}$.
\end{corollary}

\begin{proof}
 As a right adjoint to an exact functor, $U$ maps projectives to 
 projectives. 
\end{proof}

\begin{lemma}
 \label{P_is_acyclic}
 For any $\underline{A}$-module $M$, $E(M)$ is $\Sigma$-acyclic. 
\end{lemma}

\begin{proof}
 By Lemma \ref{suspension_and_Hom}, we have
 \[
  \Sigma^{n}E(M) = E(\Sigma^{n}M)
 \]
 and it suffices to prove the case when $n=0$.
 Let us show that both $Z(E(M))$ and
 $B(E(M))$ are isomorphic to $k\langle \varepsilon\rangle\otimes M$.

 Suppose $\varphi \in Z(E(M))$, which means that
 \[
  (h\cdot\varphi)(h') = \varepsilon(h)\varphi(h')
 \]
 or
 \[
  \varphi(S(h)h') = \varepsilon(h)\varphi(h')
 \]
 for all $h,h'\in H$.
 Take $h'=1$. Then we have
 \[
  \varphi(S(h)) = \varepsilon(h)\varphi(1)
 \]
 or
 \[
  \varphi(h) = \varepsilon(S^{-1}(h))\varphi(1) =
 \varepsilon(h)\varphi(1).  
 \]
 And we have $Z(E(M))\cong k\la \varepsilon\ra\otimes M$.

 Suppose $\varphi \in B(E(M))$. Then there exists
 $\psi\in E(M)$ such that
 \[
  \varphi(h) = \psi(S(\lambda)h)
 \]
 for all $h\in H$.
 It is immediate to verify that $S(\lambda)$ is a right integral and
 thus the right hand side is
 $\varepsilon(h)\psi(S(\lambda))$.
 Therefore $B(E(M))\cong k\la\varepsilon\ra\otimes M$.  
\end{proof}

%

\subsection{Model structures on Abelian categories}
\label{model_category}

This is a summary of Hovey's theory of Abelian model categories and
cotorsion pairs used in this paper.
Our main reference is Gillespie's survey \cite{1512.06001}.

Let us first recall the definition of cotorsion pairs introduced by Salce in 
\cite{Salce1979}. 

\begin{definition}
 Let $\bm{A}$ be an Abelian category.
 For objects $X,Y\in \bm{A}$, define
 \[
  X\perp Y \Longleftrightarrow \Ext_{\bm{A}}^{1}(X,Y)=0.
 \]
 More generally for a class $\bm{C}$ of objects, define
 \begin{align*}
  \bm{C}^{\perp} & = \set{Y\in\bm{A}_{0}}{C\perp Y \text{
  for all } C\in\bm{C}} \\
  {}^{\perp}\bm{C} & = \set{X\in\bm{A}_{0}}{X\perp C \text{
  for all } C\in\bm{C}}.
 \end{align*}
 These classes are also regarded as full subcategories.
\end{definition}

\begin{definition}
 \label{def:cotorsion_pair}
 Let $\bm{A}$ be as above.
 A \emph{cotorsion pair} on $\bm{A}$ is a pair
 $(\bm{P},\bm{I})$ of classes 
 of objects of $\bm{A}$ satisfying 
 \begin{enumerate}
  \item $\bm{P}={}^{\perp}\bm{I}$.
  \item $\bm{I}=\bm{P}^{\perp}$.
 \end{enumerate}
\end{definition}

\begin{definition}
 A cotorsion pair $(\bm{P},\bm{I})$ is said to \emph{have enough
 projectives} if, 
 for any $X$ in $\bm{A}$, there exists
 a short exact sequence of the form 
 \[
   0 \rarrow{} A \rarrow{} B \rarrow{} X \rarrow{} 0
 \]
 with $B\in \bm{P}$ and $A\in\bm{I}$.
 It is said to \emph{have enough injectives} if, for any $X$ in
 $\bm{A}$, there exists a short exact sequence of the form
\[
 0 \rarrow{} X \rarrow{} A' \rarrow{} B' \rarrow{} 0 
\]
 with $A'\in \bm{I}$ and $B'\in\bm{P}$.

 It is called \emph{complete} if it has both enough projectives and
 injectives. 
\end{definition}

\begin{example}
 \label{categorical_cotorsion_pair}
 Let $\Prj(\bm{A})$ and $\Inj(\bm{A})$ be the classes of projectives
 and injectives in an Abelian category $\bm{A}$.
 Then $(\bm{A},\Inj(\bm{A}))$ and $(\Prj(\bm{A}),\bm{A})$ are cotorsion
 pairs. 
 The former is complete precisely when $\bm{A}$ has enough injectives
 and the latter is complete precisely when $\bm{A}$ has enough
 projectives. 

 These are called \emph{categorical cotorsion pairs} by Hovey.
 Let us call the former the \emph{projective cotorsion pair} and the
 latter the \emph{injective cotorsion pair}.
\end{example}

In the case of Grothendieck Abelian categories, the completeness of
cotorsion pairs is closely related to the notion of generation of
cotorsion pairs. The following terminology is used in \cite{1005.3248}.

\begin{definition}
 Let $\bm{G}$ be a set of objects in an Abelian category $\bm{A}$.
 Then the pair $({}^{\perp}(\bm{G}^{\perp}),\bm{G}^{\perp})$ is called
 the \emph{cotorsion pair generated by $\bm{G}$}.
\end{definition}

\begin{remark}
 The pair $({}^{\perp}(\bm{G}^{\perp}),\bm{G}^{\perp})$ is always a
 cotorsion pair by definition.
 Some authors say that the cotorsion pair is \emph{cogenerated by
 $\bm{G}$} in the above situaltion. For example, this terminology is
 used by Hanno Becker in \cite{1205.4473}.
\end{remark}

The following fact can be
found as Proposition 1.2.1 in \cite{1205.4473} and is attribued to 
\cite{1005.3248}. 

\begin{proposition}
 Let $\bm{A}$ be a Grothendieck Abelian category.
 If $(\bm{D},\bm{E})$ is a cotorsion pair generated by a set $\bm{X}$,
 then the following hold:
 \begin{enumerate}
  \item The pair $(\bm{D},\bm{E})$ has enough injectives.
  \item The pair $(\bm{D},\bm{E})$ has enough projectives if and only if
	$\bm{D}$ is generating.
 \end{enumerate}
\end{proposition}

The following terminology is used in \cite{1205.4473}.

\begin{definition}
 A cotorsion pair $(\bm{D},\bm{E})$ is called \emph{small} if $\bm{D}$ is
 generated by a set and $\bm{D}$ is generating.
\end{definition}

\begin{corollary}
 \label{small_cotorsion_pair_in_Grothendieck_Abelian_category}
 If $\bm{A}$ is a Grothendieck Abelian category with enough projectives,
 then any cotorsion pair generated by a set is small. Thus it is
 complete. 
\end{corollary}

\begin{definition}
 Let $\bm{A}$ be a bicomplete Abelian category. 
 A model structure on $\bm{A}$ is called \emph{Abelian} if
 \begin{enumerate}
  \item a morphism is a cofibration if and only if it is a
	monomorphism with cofibrant cokernel,
  \item a morphism is a trivial cofibration if and only if it is a
	monomorphism with trivially cofibrant cokernel,
  \item a morphism is a fibration if and only if it is an
	epimorphism with fibrant kernel, and
  \item a morphism is a trivial fibration if and only if it is a
	epimorphism with trivially fibrant kernel.
 \end{enumerate}
\end{definition}

The following terminology is introduced by Gillespie
\cite{1512.06001}. 

\begin{definition}
 \label{def:Hovey_triple_in_Abelian_category}
 Let $\bm{A}$ be an Abelian category. A triple of subcategories
 $(\Cof,\Triv,\Fib)$ is called a
 \emph{Hovey triple} if 
 \begin{enumerate}
  \item $\Triv$ is a thick subcategory.
  \item $(\Cof,\Fib\cap \Triv)$ is a
	complete cotorsion pair. 
  \item $(\Cof\cap \Triv, \Fib)$ is a
	complete cotorsion pair. 
 \end{enumerate}
\end{definition}

\begin{theorem}[Hovey]
 \label{characterization_of_Abelian_model_structure}
 Let $\bm{A}$ be a bicomplete Abelian category. Suppose $\bm{A}$ is
 equipped with an Abelian model structure. 
 Denote the full subcategories of trivial, cofibrant, and fibrant
 objects by $\Triv$, $\Cof$, and $\Fib$, respectively. 
 Then $(\Cof,\Triv,\Fib)$ is a Hovey triple

 Conversely, given a Hovey triple $(\Cof, \Triv,\Fib)$
 there exists a unique Abelian model structure on $\bm{A}$ such that
 $\Triv$, $\Cof$, and $\Fib$ are subcategories of trivial, cofibrant,
 and fibrant objects, respectively. 
\end{theorem}

 By definition, in the Abelian model structure defined by a Hovey triple 
 $(\Cof,\Triv,\Fib)$, a morphism $f:X\to Y$ is
 \begin{itemize}
  \item a cofibration if and only if $f$ is a monomorphism and
	 $\Coker f\in \Cof$,
  \item a trivial cofibration if and only if $f$ is a monomorphism and 
	 $\Coker f\in \Cof\cap\Triv$,
  \item a fibration if and only if $f$ is an epimorphism and
	$\Ker f\in \Fib$, and
  \item a trivial fibration if and only if $f$ is an epimorphism and
	$\Ker f\in \Fib\cap\Triv$.
 \end{itemize}
 Furthermore the following characterization of weak equivalences is
 obtained by Hovey.

\begin{lemma}
 \label{weak_equivalence_in_Abelian_model_structure}
 In the Abelian model structure defined by a Hovey triple
 $(\Cof,\Triv,\Fib)$, a morphism $f:X\to Y$ is a weak equivalence if and
 only if there exist an epimorphism $p$ with $\Ker p\in \Triv$ and a
 monomorphism $i$ with $\Coker i \in \Triv$ such that
 $f=p\circ i$. 
\end{lemma}

By the projective cotorsion pair (Example
\ref{categorical_cotorsion_pair}) and Corollary
\ref{small_cotorsion_pair_in_Grothendieck_Abelian_category}, we obtain
the following. 

\begin{lemma}
 \label{Abelian_Hovey_triple_generated_by_Triv}
 Let $\bm{A}$ be a Grothendieck Abelian category having enough projectives.
 Given a thick subcategory $\Triv$,
 $({}^{\perp}\Triv,\Triv,\bm{A})$ is a Hovey triple
 if and only if
 \begin{enumerate}
  \item ${}^{\perp}\Triv\cap\Triv=\Prj(\bm{A})$,
	and
  \item $\left({}^{\perp}\Triv\right)^{\perp}=\Triv$,
 \end{enumerate}
 where $\Prj(\bm{A})$ is the full subcategory of projective objects in
 $\bm{A}$. 
\end{lemma}

\section{Cotorsion pairs in the category of equivariant modules}
\label{cotorsion_pair}

\subsection{Equivariant projective modules}
\label{projectives}

In \cite{Avramov-Foxby-Halperin}, Avramov, Foxby, and Halperin compared
various notions of projectives in the category of dg modules over a dg
algebra.
In this section, we study Hopfological counterparts.
Throughout this section, we fix a left $H$-module category $A$.
Given an $H$-equivariant $A$-module $P$, we have two functors
\begin{align*}
 \bm{C}_{A,H}(P,-) & : \bm{C}_{A,H} \rarrow{} \lMod{H} \\
 \bm{C}_{A,H}^{H}(P,-) & : \bm{C}_{A,H}^{H} \rarrow{} \lMod{k}.
\end{align*}
Various notions of projectivities are defined by the degrees of
preservations of surjectivities by these functors.

\begin{definition}
 \label{def:equivariant_projectives}
 Let $P$ be an $H$-equivariant left $A$-module. 
 \begin{enumerate}
  \item $P$ is called \emph{$\Sigma$-linearly projective}, if, for any
	surjective morphism $f:M\to N$, 
	\[
	\bm{C}_{A,H}(P,\Sigma^{n}f) :
	\bm{C}_{A,H}(P,\Sigma^{n}M) \rarrow{} 
	\bm{C}_{A,H}(P,\Sigma^{n}N) 
	\]
	is surjective for all $n\in\Z$.

  \item $P$ is called \emph{$\Sigma$-homotopically projective}, if, 
	for any $\Sigma$-quism $f:M\to N$,
	\[
	\bm{C}_{A,H}(P,\Sigma^{n}f) :
	\bm{C}_{A,H}(P,\Sigma^{n}M) \rarrow{} 
	\bm{C}_{A,H}(P,\Sigma^{n}N) 
	\]
	is a quism for all $n\in\Z$.

  \item $P$ is called \emph{$\Sigma$-semiprojective}, if, for any
	surjective $\Sigma$-quism $f:M\to N$,
	\[
	\bm{C}_{A,H}(P,\Sigma^{n}f) :
	\bm{C}_{A,H}(P,\Sigma^{n}M) \rarrow{} 
	\bm{C}_{A,H}(P,\Sigma^{n}N) 
	\]
	is a surjective quism for all $n\in\Z$.
	
  \item $P$ is called \emph{$\Sigma$-Qi-projective}, if
	for any surjective $\Sigma$-quism $f:M\to N$, 
	\[
	\bm{C}_{A,H}^{H}(P,\Sigma^{n}f) :
	\bm{C}_{A,H}^{H}(P,\Sigma^{n}M) \rarrow{} 
	\bm{C}_{A,H}^{H}(P,\Sigma^{n}N)
	\]
	is surjective for all $n\in\Z$.
 \end{enumerate}

 The full subcategories of $\bm{C}_{A,H}^{H}$ consisting of
 $\Sigma$-homotopically projectives, $\Sigma$-semiprojectives, and
 $\Sigma$-Qi-projectives are denoted
 by $\HoPrj_{\Sigma}$, $\SemiPrj_{\Sigma}$, and
 $\QiPrj_{\Sigma}$ respectively. 
 Corresponding full subcategories of $\cT_{A,H}^{H}$ are denoted by the
 same symbols. 
\end{definition}

\begin{remark}
 Under the isomorphism
 $\bm{C}_{A,H}(P,\Sigma^{n}M)\cong\Sigma^{n}\bm{C}_{A,H}(P,M)$,
 $P$ is $\Sigma$-homotopically projective if and only if
 $\bm{C}_{A,H}(P,-)$
 transforms $\Sigma$-quisms to $\Sigma$-quisms.
\end{remark}

\begin{lemma}
 \label{Qi-projective_is_A-projective}
 If $P$ is either $\Sigma$-semiprojective or $\Sigma$-Qi-projective,
 then $U(P)$ is projective in $\lMod{\underline{A}}$. 
\end{lemma}

\begin{proof}
 Suppose we have a diagram of $\underline{A}$-modules
 \[
 \xymatrix{
  & U(P) \ar[d]_{g} & \\
  M \ar[r]_{f} & N \ar[r] & 0
 }
 \]
 in which the bottom row is exact.
 By taking the right adjoint, we obtain a diagram
 \[
 \xymatrix{
 & P \ar[d]^{\tilde{g}} & \\
 E(M) \ar[r]^{E(f)} & E(N) \ar[r] & 0 
 }
 \]
 by Proposition \ref{P_is_right_adjoint_to_U}.
 Since $E$ is an exact functor, the bottom row is exact.
 By Lemma \ref{P_is_acyclic}, both $E(M)$  and $E(N)$
 are $\Sigma$-acyclic. Hence $E(f)$ is a surjective $\Sigma$-quism. 

 When $P$ is $\Sigma$-Qi-projective, the induced map
 \[
 \bm{C}_{A,H}^{H}(P,E(f)) : \bm{C}_{A,H}^{H}(P,E(M)) \rarrow{}
 \bm{C}_{A,H}^{H}(P,E(N)) 
 \]
 is surjective by definition.

 When $P$ is $\Sigma$-semiprojective, the induced map
 \[
 \bm{C}_{A,H}(P,E(f)) : \bm{C}_{A,H}(P,E(M)) \rarrow{}
 \bm{C}_{A,H}(P,E(N)) 
 \]
 is a surjective quism. 
 By Lemma \ref{surjectivity}, $\bm{C}_{A,H}^{H}(P,f)$ is surjective. 

 Thus $\tilde{g}$ has a lift $P\to E(M)$ in
 $\bm{C}_{A,H}^{H}$ in both cases. 
 The left adjoint to this morphism is a lift of $g$.
 Hence $U(P)$ is projective as an $\underline{A}$-module in both cases.
\end{proof}

The following is a modification of Lemma 6.2 in Qi's paper.

\begin{lemma}
 \label{Qi-projective_is_stably_trivial}
 If $P$ is $\Sigma$-Qi-projective, then, for any $\Sigma$-acyclic object
 $T$, $\bm{C}_{A,H}(P,T)$ is $\Sigma$-acyclic. In other words,
 $\cT_{A,H}^{H}(P,\Sigma^{n}T)=0$ for all $n$.
\end{lemma}

\begin{proof}
 If $T$ is $\Sigma$-acyclic, the surjective map
 \[
  1\otimes \varepsilon : T\otimes H \rarrow{} T
 \]
 is a $\Sigma$-quism, since $T\otimes H=C(T)$ is $\Sigma$-acyclic by
 Corollary \ref{suspension_of_cone_is_acyclic}.
 Since $P$ is $\Sigma$-Qi-projective, the induced map
 \[
  \bm{C}_{A,H}^{H}(P,\Sigma^{n}(T\otimes H)) \rarrow{}
 \bm{C}_{A,H}^{H}(P,\Sigma^{n}(T)) 
 \]
 is surjective for all $n\in\Z$.
 For any $\varphi\in\bm{C}_{A,H}^{H}(P,\Sigma^{n}T)$, there exists
 $\psi\in\bm{C}_{A,H}^{H}(P,\Sigma^{n}(T\otimes H))$ such that
 $(1\otimes\varepsilon)\circ\psi=\varphi$.
 By Lemma \ref{suspension_and_Hom}, we have an isomorphism of
 $H$-modules 
 \[
  \bm{C}_{A,H}(P,\Sigma^{n}(T\otimes H)) \cong
 \Sigma^{n}(\bm{C}_{A,H}(P,T)\otimes H). 
 \]
 By Corollary \ref{suspension_of_cone_is_acyclic}, this 
 is acyclic, which implies that there exists
 $\rho\in \bm{C}_{A,H}(P,\Sigma^{n}(T\otimes H))$ such that
 $\psi=\lambda\rho$.
 By the same calculation as in the proof of Lemma 6.2 of Qi's paper, we
 have
 \[
  \varphi = \lambda(1\otimes\varepsilon)\circ\rho.
 \]
 Thus $\cT_{A,H}(P,\Sigma^{n}T)=0$ for all $n$.
\end{proof}

Avramov, Foxby, and Halperin \cite{Avramov-Foxby-Halperin} found many
equivalent descriptions of homotopical projectivity and
semiprojectivity. Here we prove analogues of some of them. 

\begin{lemma}
 \label{Sigma-semiprojectivity}
 For an $H$-equivariant $A$-module $P$, the following
 conditions are equivalent:
 \begin{enumerate}
  \item \label{condition:semiprojective}
	$P$ is $\Sigma$-semiprojective.
  \item \label{condition:dg_projective}
	$P$ is $\Sigma$-homotopically projective and
	$U(P)$ is projective as an $A$-module.
  \item \label{condition:Qi-projective}
	$P$ is $\Sigma$-Qi-projective.
 \end{enumerate}
\end{lemma}

\begin{proof}
 Let us first show that (\ref{condition:semiprojective}) and
 (\ref{condition:dg_projective}) are equivalent.
 
 Suppose $P$ is $\Sigma$-semiprojective.
 By Lemma \ref{Qi-projective_is_A-projective}, $P$ is projective as an
 $\underline{A}$-module. 
 In order to show that $P$ is $\Sigma$-homotopically projective, let
 $f:M\to N$ be a $\Sigma$-quism.
 We have a short exact sequence
 \[
  0 \rarrow{} \Sigma^n(N) \rarrow{} C_{\Sigma^{n}f} \rarrow{\delta}
 \Sigma^{n+1}(M) \rarrow{} 0, 
 \]
 which defines a triangle
 \[
  \Sigma^{n}(M) \rarrow{\Sigma^{n}f} \Sigma^n(N) \rarrow{}
 C_{\Sigma^{n}f} \rarrow{} \Sigma^{n+1}(M). 
 \]
 Since $H$ is a homological functor, we see from the long exact sequence
 associated with this triangle that $H(C_{\Sigma^{n}f})=0$ for all $n$.
 In other words, 
 \[
  C_{f} \rarrow{} 0
 \]
 is a surjective $\Sigma$-quism.
 Since $P$ is $\Sigma$-semiprojective,
 \[
  \bm{C}_{A,H}(P,C_{f})) \rarrow{} \bm{C}_{A,H}(P,0) = 0
 \]
 is a surjective $\Sigma$-quism, i.e.~$\bm{C}_{A,H}(P,C_{f})$ is
 $\Sigma$-acyclic. 

 On the other hand, since $P$ is projective as an $\underline{A}$-module,
 $\bm{C}_{A,H}(P,-)=(\lMod{\underline{A}})(P,-)$ is an exact functor 
 and we obtain an extension
 \[
 0 \rarrow{} \bm{C}_{A,H}(P,\Sigma^{n}(N)) \rarrow{}
 \bm{C}_{A,H}(P,C_{\Sigma^{n}f}) 
 \rarrow{\delta} \bm{C}_{A,H}(P,\Sigma^{n+1}(M)) \rarrow{} 0 
 \]
 in $\lMod{H}$.
 By Lemma \ref{mapping_cone_and_A-hom} and \ref{suspension_and_Hom}, this
 is isomorphic to  
 \[
 0 \rarrow{} \Sigma^{n}\bm{C}_{A,H}(P,N) \rarrow{}
 C_{\bm{C}_{A,H}(P,\Sigma^{n}f)} 
 \rarrow{\delta} \Sigma^{n+1}(\bm{C}_{A,H}(P,M)) \rarrow{} 0, 
 \]
 and thus we obtain a triangle
 \[
 \bm{C}_{A,H}(P,\Sigma^{n}(M)) \rarrow{\bm{C}_{A,H}(P,\Sigma^{n}(f))}
 \bm{C}_{A,H}(P,\Sigma^{n}(N)) \rarrow{} 
 \bm{C}_{A,H}(P,C_{\Sigma^n f}) 
 \rarrow{\delta} \bm{C}_{A,H}(P,\Sigma^{n+1}(M))
 \]
 in the homotopy category $\ho(\lMod{H})$.

 By the associated long exact sequence of homology, we see that
 \[
 H(\bm{C}_{A,H}(P,\Sigma^{n}(f))): H(\bm{C}_{A,H}(P,\Sigma^{n}(M)))
 \rarrow{} H(\bm{C}_{A,H}(P,\Sigma^{n}(N)))
 \]
 is an isomorphism for all $n$, since $\bm{C}_{A,H}(P,C_{f})$ is
 $\Sigma$-acyclic. 
 And thus $P$ is homotopically $\Sigma$-semiprojective.
 
 Conversely suppose that $U(P)$ is projective as an
 $\underline{A}$-module and $P$ 
 is $\Sigma$-homotopically semiprojective.
 Let $f: M\to N$ be a surjective $\Sigma$-quism.
 Since $P$ is $\Sigma$-homotopically projective, $\bm{C}_{A,H}(P,f)$ is a
 $\Sigma$-quism. Furthermore, since $U(P)$ is projective in
 $\lMod{\underline{A}}$, 
 $\bm{C}_{A,H}(P,-)$ is an exact functor. In particular,
 $\bm{C}_{A,H}(P,f)$ is surjective.
 Hence $P$ is $\Sigma$-semiprojective.

 We next show that (\ref{condition:semiprojective}) and
 (\ref{condition:Qi-projective}) are equivalent.
 Suppose $P$ is $\Sigma$-semiprojective and let $f:M\to N$ be a
 surjective $\Sigma$-quism.
 By definition, 
 \[
  \bm{C}_{A,H}(P,\Sigma^{n}f) : \bm{C}_{A,H}(P,\Sigma^{n}M) \rarrow{}
 \bm{C}_{A,H}(P,\Sigma^{n}N) 
 \]
 is a surjective quism for all $n\in\Z$. By Lemma \ref{surjectivity},
 the induced map
 \[
 \bm{C}_{A,H}^{H}(P,\Sigma^{n}f):
 \bm{C}_{A,H}^{H}(P,\Sigma^{n}M) =
  Z(\bm{C}_{A,H}(P,\Sigma^{n}M))
  \rarrow{Z(\bm{C}_{A,H}(P,\Sigma^{n}f))}
  Z(\bm{C}_{A,H}(P,\Sigma^{n}N)) = \bm{C}_{A,H}^{H}(P,\Sigma^{n}N) 
 \]
 is surjective. Hence $P$ is $\Sigma$-Qi-projective.

 Conversely suppose that $P$ is $\Sigma$-Qi-projective.
 Let $f:M\to N$ be a surjective $\Sigma$-quism.
 We need to show that the induce morphism
 \[
 \bm{C}_{A,H}(P,\Sigma^{n}f): \bm{C}_{A,H}(P,\Sigma^{n}M) \rarrow{}
 \bm{C}_{A,H}(P,\Sigma^{n}N)
 \]
 is a surjective quism for all $n\in\Z$.
 By Lemma \ref{Qi-projective_is_A-projective}, $U(P)$ is a projective
 $\underline{A}$-module and thus $\bm{C}_{A,H}(P,-)$ is an exact
 functor, which implies that $\bm{C}_{A,H}(P,\Sigma^{n}f)$ is
 surjective.

 Let us show that $\bm{C}_{A,H}(P,\Sigma^{n}f)$ is a quism for all
 $n$, i.e.~the induced map 
 \[
 H(\bm{C}_{A,H}(P,\Sigma^{n}M))
 \cong \cT_{A,H}^{H}(P, \Sigma^{n} M)
 \rarrow{\cT_{A,H}^{H}(P,\Sigma^{n}f)} 
 \cT_{A,H}^{H}(P,\Sigma^{n} N) \cong H(\bm{C}_{A,H}(P,\Sigma^{n}N))
 \]
 is an isomorphism for all $n\in\Z$.

 Under the identification in Lemma \ref{mapping_cone_and_A-hom},
 the short exact sequence
 \[
  0 \rarrow{} \Sigma^{n}N \rarrow{} \Sigma^{n}C_{f} \rarrow{}
 \Sigma^{n+1} M \rarrow{} 0
 \]
 defines a triangle in $\cT_{A,H}^{H}$ and thus we obtain a long exact
 sequence
 \[
  \cdots \rarrow{} \cT_{A,H}^{H}(P,\Sigma^{n-1}C_{f}) \rarrow{}
 \cT_{A,H}^{H}(P,\Sigma^{n}M) \rarrow{} \cT_{A,H}^{H}(P,\Sigma^{n}N)
 \rarrow{}  \cT_{A,H}^{H}(P,\Sigma^{n}C_{f}) \rarrow{} \cdots.
 \]
 Since $f : M \to N$ is a $\Sigma$-quism, $C_{f}$ is $\Sigma$-acyclic.
 Since $P$ is $\Sigma$-Qi-projective, by Lemma
 \ref{Qi-projective_is_stably_trivial}, we have
 $\cT_{A,H}^{H}(P,\Sigma^{n}C_{f})=0$.
 Therefore $\cT_{A,H}^{H}(P,\Sigma^{n}f)$ is an isomorphism for all $n$.
\end{proof}

\begin{corollary}
 An object $P$ in $\bm{C}_{A,H}^{H}$ is $\Sigma$-semiprojective if and
 only if $U(P)$ is projective in $\lMod{\underline{A}}$ and 
 $\bm{C}_{A,H}(P,T)\in\Triv^{\Sigma}_{A,H}$ for
 all $T\in\Triv^{\Sigma}_{A,H}$.
\end{corollary}

\begin{proof}
 Suppose $P$ is $\Sigma$-semiprojective. For
 $T\in\Triv^{\Sigma}_{A,H}$, $T\to 0$ is a surjective $\Sigma$-quism and
 hence induces a $\Sigma$-quism $\bm{C}_{A,H}(P,T)\to 0$.

 Conversely, suppose $U(P)$ is projective and that
 $\bm{C}_{A,H}(P,T)$ is $\Sigma$-acyclic for all
 $T\in\Triv^{\Sigma}_{A,H}$.
 By the previous Lemma, it suffices to show that $P$ is
 $\Sigma$-homotopically projective.
 For a $\Sigma$-quism $f:M\to N$, $C_{f}$ is $\Sigma$-acyclic.
 Since $U(P)$ is projective, we have a triangle
 \[
 \bm{C}_{A,H}(P,M) \rarrow{\bm{C}_{A,H}(P,f)} \bm{C}_{A,H}(P,N)
 \rarrow{\bm{C}_{A,H}(P,j_{f})} 
 \bm{C}_{A,H}(P,C_{f}) \rarrow{\bm{C}_{A,H}(P,\delta_{f})}
 \bm{C}_{A,H}(P,\Sigma(M)). 
 \]
 By assumption $\bm{C}_{A,H}(P,C_{f})$ is $\Sigma$-acyclic. The long
 exact sequence of homology implies that $\bm{C}_{A,H}(P,f)$ is a
 $\Sigma$-quism. 
\end{proof}

%

\subsection{The orthogonality in the category of equivariant modules}
\label{orthogonality}

In this section, we study the orthogonality in $\bm{C}_{A,H}^{H}$ with
respect to the biadditive functor $\Ext^{1}_{\bm{C}_{A,H}^{H}}(-,-)$.
For objects $X,Y\in \bm{C}_{A,H}^{H}$, we denote
$X\perp Y$ when $\Ext_{\bm{C}_{A,H}^{H}}^{1}(X,Y)=0$.
For a class $\bm{T}$ of objects in $\bm{C}_{A,H}^{H}$, define
\begin{align*}
 \bm{T}^{\perp} & = \set{Y}{C\perp Y \text{
 for all } C\in\bm{T}} \\
 {}^{\perp}\bm{T} & = \set{X}{X\perp C \text{
 for all } C\in\bm{T}}.
\end{align*}
We begin with the following general fact.

\begin{lemma}
 \label{Ext_of_suspension}
 If $\bm{T}$ is a class of objects in $\bm{C}_{A,H}^{H}$ closed under
 $\Sigma^{n}$ for all $n\in\Z$, then so are both 
 ${}^{\perp}\bm{T}$ and $\bm{T}^{\perp}$.
\end{lemma}

\begin{proof}
 We first prove that ${}^{\perp}\bm{T}$ is closed under $\Sigma$.
 Suppose $M\in {}^{\perp}\bm{T}$. We need to prove
 $\Ext^{1}_{\bm{C}_{A,H}^{H}}(\Sigma M, N)=0$ for all 
 $N\in\bm{T}$.
 Let
 \[
  0 \rarrow{} N \rarrow{} E \rarrow{} \Sigma M \rarrow{} 0
 \]
 be an extension in $\bm{C}_{A,H}^{H}$.
 Since $\Sigma^{-1}$ is an exact functor, we obtain an extension
 \[
  0 \rarrow{} \Sigma^{-1} N \rarrow{} \Sigma^{-1} E \rarrow{}
 \Sigma^{-1}\Sigma M \rarrow{} 0. 
 \]
 Let
 \[
  i : M \rightarrow M\oplus M\otimes Q \cong \Sigma^{-1}\Sigma M
 \]
 be the inclusion under the isomorphism used in the proof of Lemma
 \ref{Sigma}.
 By taking the pullback along $i$, we obtain an extension
 \[
  0 \rarrow{} \Sigma^{-1} N \rarrow{} i^{*}\Sigma^{-1} E \rarrow{}
  M \rarrow{} 0. 
 \]
 Since $\bm{T}$ is closed under $\Sigma^{-1}$, this sequence splits in
 $\bm{C}_{A,H}^{H}$. 
 Let $s: M\to i^{*}\Sigma^{-1} E$ be a splitting.
 The composition
 \[
  \Sigma M \rarrow{\Sigma s} \Sigma i^{*}\Sigma^{-1} E \rarrow{}
 \Sigma\Sigma^{-1} E \cong E\oplus E\otimes Q \rarrow{} E
 \]
 gives us a splitting we wanted by the commutativity of the following
 diagram
 \[
 \xymatrix{
  0 \ar[r] & \Sigma\Sigma^{-1} N \ar[d] \ar[r]
  & \Sigma i^{*}\Sigma^{-1}E \ar[d] \ar[r] & \Sigma M
  \ar[d] \ar[r] & 0 \\  
  0 \ar[r] & \Sigma\Sigma^{-1} N \ar[d] \ar[r]
  & \Sigma\Sigma^{-1} E \ar[d] \ar[r] & \Sigma\Sigma^{-1}\Sigma M \ar[d]
 \ar[r] & 0 \\
  0 \ar[r] & N \ar[r] & E \ar[r] & \Sigma M \ar[r] & 0.
 }
 \] 
 The same argument can be used to show that ${}^{\perp}\bm{T}$ is closed
 under $\Sigma^{-1}$ by switching $\Sigma$ and $\Sigma^{-1}$. The detail
 is omitted.

 We next show that $\bm{T}^{\perp}$ is closed under $\Sigma$.
 The argument is essentially dual to the above case.
 Suppose $N\in \bm{T}^{\perp}$. We need to show that
 $\Ext^{1}_{\bm{C}_{A,H}^{H}}(M,\Sigma N)=0$ for all $M\in\bm{T}$.
 Let
 \begin{equation}
  0 \rarrow{} \Sigma N \rarrow{} E \rarrow{} M \rarrow{} 0
   \label{equation:extension_SigmaN_M}
 \end{equation}
 be an extension. Apply $\Sigma^{-1}$ to obtain an extension
 \[
  0 \rarrow{} \Sigma^{-1}\Sigma N \rarrow{} \Sigma^{-1}E \rarrow{}
 \Sigma^{-1}M \rarrow{} 0. 
 \]
 Let $p: \Sigma^{-1}\Sigma N\cong N\oplus N\otimes Q \to N$ be the
 projection. By taking the pushout along $p$, we obtain an extension
 \[
  0 \rarrow{} N \rarrow{} p_{*}\Sigma^{-1}E \rarrow{}
 \Sigma^{-1}M \rarrow{} 0. 
 \]
 Since $\bm{T}$ is closed under $\Sigma^{-1}$, this extension splits. 
 Let $q: p_{*}\Sigma^{-1}E\to N$ be a splitting. Then the composition
 \[
  E \rarrow{} \Sigma\Sigma^{-1}E \rarrow{} \Sigma p_{*}\Sigma^{-1}E
 \rarrow{\Sigma q} \Sigma N
 \]
 is a splitting of (\ref{equation:extension_SigmaN_M}). And we have
 $\Ext^{1}_{\bm{C}_{A,H}^{H}}(M,\Sigma N)=0$ for all $M\in\bm{T}$.
 The case of $\Sigma^{-1}$ on $\bm{T}^{\perp}$ is analogous and is
 omitted. 
\end{proof}

\begin{corollary}
 \label{from_perp_to_null_homotopic}
 Let $\bm{T}$ be a class of objects closed under $\Sigma^{n}$ for all
 $n\in\Z$.
 Suppose $P\in {}^{\perp}\bm{T}$ and $T\in \bm{T}$.
 Then $\cT_{A,H}^{H}(P,\Sigma^{n}T)=0$ for all $n\in\Z$.
\end{corollary}

\begin{proof}
 Let $T$ be an object in $\bm{T}$.
 Under the identification
 \[
  \cT_{A,H}^{H}(P,\Sigma^{n}T) = \bm{C}_{A,H}^{H}(P,\Sigma^{n}T)/_{\simeq},
 \]
 it suffices to show that any morphism $\varphi : P\to\Sigma^{n}T$ is
 null homotopic.
 By Lemma \ref{Ext_of_suspension}, $\Sigma P\in {}^{\perp}\bm{T}$.
 In particular, the extension
 \[
  0 \rarrow{} \Sigma^{n}T \rarrow{} C_{\varphi} \rarrow{} \Sigma P
 \rarrow{} 0
 \]
 splits. By Corollary \ref{mapping_cone_of_null_homotopic_map}, 
 $\varphi\simeq 0$. 
\end{proof}

Recall from Lemma \ref{Triv*_is_thick} that
the subcategory $\Triv_{A,H}^{\Sigma}$ of $\Sigma$-acyclic objects is a
thick subcategory of $\bm{C}_{A,H}^{H}$. 
By the identification $\bm{C}_{A,H}^{H}\cong \lMod{(A\# H)}$ of
Proposition \ref{equivariant_module_as_module_over_smash_product},
$\bm{C}_{A,H}^{H}$ is an Abelian category with enough projectives.
Thus, by Lemma \ref{Abelian_Hovey_triple_generated_by_Triv}, the triple
$({}^{\perp}\Triv_{A,H}^{\Sigma},\Triv_{A,H}^{\Sigma},\bm{C}_{A,H}^{H})$ is
a Hovey triple if 
\begin{align*}
 {}^{\perp}\Triv_{A,H}^{\Sigma}\cap\Triv_{A,H}^{\Sigma} & =
 \category{Prj}(\bm{C}_{A,H}^{H}) \\
 \left({}^{\perp}\Triv_{A,H}^{\Sigma}\right)^{\perp} & =
 \Triv_{A,H}^{\Sigma}. 
\end{align*}

In order to prove these statements, we first need to understand 
${}^{\perp}\Triv_{A,H}^{\Sigma}$.

\begin{proposition}
 \label{orthogonal_to_TrivH}
 We have
 \[
 {}^{\perp}\Triv_{A,H}^{\Sigma} =
 \SemiPrj_{\Sigma}. 
 \]
\end{proposition}

\begin{proof}
 Let $P$ be an object of
 ${}^{\perp}\Triv_{A,H}^{\Sigma}$.
 Let us first show that $U(P)$ is projective as an
 $\underline{A}$-module.
 Consider a diagram of $\underline{A}$-modules
 \[
 \xymatrix{
  &  U(P) \ar[d]^{\tilde{f}} & \\
  M \ar[r]^{\varphi} & N \ar[r] & 0.
 }
 \]
 By taking the right adjoint, we have a diagram of
 $H$-modules
 \[
 \xymatrix{
  & P \ar[d]^{f} & \\
  E(M) \ar[r]^{E(\varphi)} & E(N) \ar[r] & 0
 }
 \]
 by Lemma \ref{P_is_right_adjoint_to_U}.
 Since $E$ is an exact functor, the bottom sequence is exact.
 By taking the pullback, we obtain a map of short exact sequences in
 $\bm{C}_{A,H}^{H}$ 
 \[
 \xymatrix{
  0 \ar[r] & E(\Ker\varphi) \ar@{=}[d]
  \ar[r] & E(M)\times_{E(N)} P \ar[d]
  \ar[r] & P 
  \ar[d]^{\tilde{f}} \ar[r] & 0 \\ 
  0 \ar[r] & E(\Ker\varphi) \ar[r] & E(M) \ar[r]^{E(\varphi)}
  & E(N) \ar[r] & 0. 
 }
 \]

 By Lemma \ref{P_is_acyclic}, $E(\Ker\varphi)$ belongs to
 $\Triv_{A,H}^{\Sigma}$, which implies
 \[
 \Ext^{1}_{\bm{C}_{A,H}^{H}}(P,E(\Ker\varphi))=0 
 \]
 by assumption. In other words, the top row splits in
 $\bm{C}_{A,H}^{H}$. 
 Let $s:P\to E(M)\times_{E(N)}P$ be a
 splitting. Then the adjoint to the composition
 \[
 P \rarrow{s} E(M)\times_{E(N)}P \rarrow{}
 E(M) 
 \]
 gives us a lift $U(P)\to M$ of $f$.
 Hence $U(P)$ is projective as an $\underline{A}$-module.

 By Lemma \ref{Sigma-semiprojectivity}, it remains to show that
 $P$ is $\Sigma$-homotopically projective.
 Let $f:M\to N$ be a $\Sigma$-quism in $\bm{C}_{A,H}^{H}$.
 By Lemma \ref{mapping_cone_and_A-hom}, It suffices to show that
 $\bm{C}_{A,H}(P,\Sigma^{n}C_{f})$ is acyclic for any $n\in\Z$.
 Under the identification
 \[
  H(\bm{C}_{A,H}(P,\Sigma^{n}C_{f}))\cong
 \cT_{A,H}^{H}(P,\Sigma^{n}C_{f}) = 
 \bm{C}_{A,H}^{H}(P,\Sigma^{n}C_{f})/_{\simeq}, 
 \]
 we are going to show that any morphism $\varphi: P\to\Sigma^{n}C_{f}$ is
 null homotopic.
 Consider the extension associated with the mapping cone of $\varphi$
 \[
  0 \rarrow{} \Sigma^{n}C_{f} \rarrow{} C_{\varphi} \rarrow{} \Sigma(P)
 \rarrow{} 0. 
 \]
 By assumption, $\Sigma^{n}C_{f}$ is acyclic for all $n$. 
 Since $P$ belongs to ${}^{\perp}\Triv_{A,H}^{\Sigma}$, so does $\Sigma P$
 by Lemma \ref{Ext_of_suspension}, which implies that this extension
 splits. 
 By Lemma \ref{mapping_cone_of_null_homotopic_map}, we have $f\simeq 0$.

 Conversely, suppose that $P$ is $\Sigma$-semiprojective.
 For an extension
 \begin{equation}
  0 \rarrow{} T \rarrow{j} E \rarrow{p} P \rarrow{} 0
   \label{equation:extension_by_Triv}
 \end{equation}
 with $T\in\Triv_{A,H}^{\Sigma}$, since $P$ is projective as an
 $\underline{A}$-module, 
 it defines a distinguished triangle
 \[
  T \rarrow{[j]} E \rarrow{[p]} P \rarrow{\delta} \Sigma T
 \]
 in $\cT_{A,H}^{H}$ by Lemma \ref{Hopf_mapping_cone}.
 The induced long exact sequence of homology
 \[
 \cdots \rarrow{} H(\Sigma^{n}T) \rarrow{j_{*}} H(\Sigma^{n}E)
 \rarrow{p_{*}} 
 H(\Sigma^{n}P) \rarrow{\delta} H(\Sigma^{n+1} T) \rarrow{}
 H(\Sigma^{n+1}E) \rarrow{} \cdots
 \]
 and the assumption that $T\in\Triv_{A,H}^{\Sigma}$ implies that $p:E\to P$
 is a surjective $\Sigma$-quism.
 
 By Lemma \ref{Sigma-semiprojectivity}, the induced map
 \[
 \bm{C}_{A,H}^{H}(P,p): \bm{C}_{A,H}^{H}(P,E) \rarrow{}
 \bm{C}_{A,H}^{H}(P,P) 
 \]
 is surjective. Any $s\in\bm{C}_{A,H}^{H}(P,E)$ with
 $\bm{C}_{A,H}^{H}(P,p)(s)=1_{P}$ is a splitting of
 (\ref{equation:extension_by_Triv}), which implies that
 $P\in{}^{\perp}\Triv_{A,H}^{\Sigma}$.
\end{proof}

\begin{corollary}
 The subcategory $\SemiPrj_{\Sigma}$ is closed under
 $\Sigma^{n}$ for all $n\in\Z$. 
\end{corollary}

\begin{proposition}
 We have
 \[
  {}^{\perp}\Triv_{A,H}^{\Sigma} \cap \Triv_{A,H}^{\Sigma} =
 \Prj(\bm{C}_{A,H}^{H}), 
 \]
 where the right hand side is the full subcategory of projective objects
 in $\bm{C}_{A,H}^{H}$.
\end{proposition}

\begin{proof}
 Since $\Triv_{A,H}^{\Sigma}\subset \bm{C}_{A,H}^{H}$, we have
 \[
  {}^{\perp}\Triv_{A,H}^{\Sigma} \supset {}^{\perp}\bm{C}_{A,H}^{H} =
 \Prj(\bm{C}_{A,H}^{H}). 
 \]
 Let $P$ be a projective object in $\bm{C}_{A,H}^{H}$. 
 We show that $P$ is $\Sigma$-acyclic. 
 Under the identification
 \[
  \bm{C}_{A,H}^{H} \cong \lMod{A\# H},
 \]
 $P$ is a direct summand of a free $A\# H$-module, i.e.~there exists a
 free $\underline{A}$-module $F$ such that $P$ is a direct summand of $C(F)$.
 By Corollary \ref{suspension_of_cone_is_acyclic}, $C(F)$ belongs to
 $\Triv_{A,H}^{\Sigma}$, and so does $P$. 
 And we have
 \[
  \Prj (\bm{C}_{A,H}^{H}) \subset {}^{\perp}\Triv_{A,H}^{\Sigma} \cap
 \Triv_{A,H}^{\Sigma}. 
 \]

 Conversely suppose that
 \[
 P\in {}^{\perp}\Triv_{A,H}^{\Sigma}\cap\Triv^{H}_\Sigma  =
 \SemiPrj_{\Sigma}\cap\Triv_{A,H}^{\Sigma}. 
 \]
 Since $\bm{C}_{A,H}^{H}\cong\lMod{A\# H}$ has enough projectives, we may
 take a projective object $Q$ in $\bm{C}_{A,H}^{H}$ and a surjection
 \[
  Q \rarrow{\varphi} P \rarrow{} 0.
 \]
 Since $P$ is $\Sigma$-acyclic, this is a surjective $\Sigma$-quism.
 By Lemma \ref{Sigma-semiprojectivity}, $P$ is
 $\Sigma$-Qi-projective and the induced map
 \[
 \bm{C}_{A,H}^{H}(P,f) : \bm{C}_{A,H}^{H}(P,Q) \rarrow{}
 \bm{C}_{A,H}^{H}(P,P)  
 \]
 is surjective, which implies that $Q\rarrow{p}P$ has a section.
 As a direct summand of a projective object, $P$ is projective.
%
\end{proof}

\begin{proposition}
 We have
 \[
  \left({}^{\perp}\Triv_{A,H}^{\Sigma}\right)^{\perp} = \Triv_{A,H}^{\Sigma}. 
 \]
\end{proposition}

\begin{proof}
 By definition, we have an inclusion
 \[
 \Triv_{A,H}^{\Sigma}\subset
 \left({}^{\perp}\Triv_{A,H}^{\Sigma}\right)^{\perp}. 
 \]
 We prove
 \[
  \left(\SemiPrj_{\Sigma}\right)^{\perp} =
 \left({}^{\perp}\Triv_{A,H}^{\Sigma}\right)^{\perp} \subset 
 \Triv_{A,H}^{\Sigma}.  
 \]

 Suppose $T\in \left(\SemiPrj_{\Sigma}\right)^{\perp}$. Then
 for any $P\in \SemiPrj_{\Sigma}$,
 $\Ext^{1}_{\bm{C}_{A,H}^{H}}(P,T)=0$.
 Let $P=A\otimes k$, where $k$ is regarded as a trivial $H$-module.
 The $\underline{A}$-module structure is given by the composition of
 morphisms of $A$.
 Note that $A\otimes k$ is $\Sigma$-semiprojective, since for any
 $\Sigma$-quism $f:M\to N$, the commutativity of the diagram
 \[
 \xymatrix{
  \bm{C}_{A,H}(P,\Sigma^{n}M) \ar[d]_{\cong}
   \ar[rr]^{\bm{C}_{A,H}(P,\Sigma^{n}f)} 
   & & \bm{C}_{A,H}(P,\Sigma^{n}N) \ar[d]^{\cong} \\ 
   (\lMod{k})(k,\Sigma^{n}M) \ar[d]_{\cong} \ar[rr]
   & & (\lMod{k})(k,\Sigma^{n}N) \ar[d]^{\cong} \\
   \Sigma^{n}M \ar[rr]_{\Sigma^{n}f}
   & & \Sigma^{n}N   
 }
 \]
 allows us to identify $\bm{C}_{A,H}(P,\Sigma^{n}f)$ with $\Sigma^{n}f$.
 
 Since $\SemiPrj_{\Sigma}$ is closed under $\Sigma^{n}$ for all
 $n\in\Z$, so does $(\SemiPrj_{\Sigma})^{\perp}$ by Lemma
 \ref{Ext_of_suspension} and we have 
 $\Ext^{1}_{\bm{C}_{A,H}^{H}}(P,\Sigma^{n}T)=0$ for all $n\in\Z$. 
 By Corollary \ref{from_perp_to_null_homotopic}, we have
 $\cT_{A,H}^{H}(P,\Sigma^{n}T)=0$ for all $n\in\Z$.
 Under the identification
 \[
 \cT_{A,H}^{H}(P,\Sigma^{n}T) = H(\bm{C}_{A,H}(P,\Sigma^{n}T)) \cong
 H((\lMod{k})(k,\Sigma^{n}T)) \cong H(\Sigma^{n}T),
 \]
 this implies that $T$ belongs to $\category{Triv}^{\Sigma}_{A,H}$.
\end{proof}

\begin{corollary}
 \label{Abelian_Hovey_triple_generated_by_TrivH}
 The triple
 $\left(\SemiPrj_{\Sigma},\Triv_{A,H}^{\Sigma},\bm{C}_{A,H}^{H}\right)$
 is a Hovey triple in $\bm{C}_{A,H}^{H}$. 
\end{corollary}

Now Theorem \ref{main_theorem:model_structure} is a corollary to this fact.
By Hovey's correspondence this model structure is described as follows.

\begin{corollary}
 \label{equivariant_projective_model_structure}
 There exists an Abelian model structure on $\bm{C}_{A,H}^{H}$ with the
 following properties:
 \begin{enumerate}
  \item A morphism $f: M\to N$ is a cofibration if and only if it is a
	monomorphism and $\Coker f$ is $\Sigma$-semiprojective.
  \item A morphism $f: M\to N$ is a fibration if and only if it is an
	epimorphism. 
  \item A morphism $f:M\to N$ is a weak equivalence if and only if it is
	a $\Sigma$-quism.
 \end{enumerate}
 In particular, all objects are fibrant and cofibrant objects are
 $\Sigma$-semiprojectives.
\end{corollary}

\begin{proof}
 It only remains to verify that weak equivalences agree with
 $\Sigma$-quisms. 
 Suppose $f$ is a weak equivalence.
 By Lemma \ref{weak_equivalence_in_Abelian_model_structure}, it factors
 as $f=p\circ i$ with $\Ker p,\Coker \in \Triv_{A,H}^{\Sigma}$. In
 particular both $p$ and $i$ are $\Sigma$-quisms and thus so is $f$.
 Conversely suppose $f:M\to N$ is a $\Sigma$-quism. By the model
 structure, it factors as $f=p\circ i$ with $i:M\to E$ a cofibration and
 $p:E\to N$ a trivial fibration. 
 By the characterization of trivial fibrations in an Abelian model
 category, $p$ is an epimorphism with $\Ker p\in\Triv_{A,H}^{\Sigma}$.
 And $i$ is a monomorphism with $\Coker i\in \SemiPrj_{\Sigma}$.
 By Lemma \ref{Sigma-semiprojectivity}, the sequence
 \[
  0 \rarrow{} M \rarrow{i} E \rarrow{} \Coker i \rarrow{} 0
 \]
 is an $A$-split sequence. Hence defines a triangle in $\cT_{A,H}^{H}$.
 Since both $f$ and $p$ are $\Sigma$-quisms, so is $i$. The long exact
 sequence of homology induced by this triangle implies that
 $\Coker i\in\Triv_{A,H}^{\Sigma}$. 
\end{proof}

\begin{definition}
 Under this model structure, the full subcategory of compact cofibrant
 objects in $\bm{C}_{A,H}^{H}$ is denoted by $\Perf_{A,H}^{H}$. Objects
 of this category are called \emph{perfect $H$-equivariant $A$-modules}.
 Cofibrations and weak equivalences in this model structure make
 $\Perf_{A,H}^{H}$ into a Waldhausen category, whose algebraic
 $K$-theory is denoted by $K(A,H)$ and is called the \emph{Hopfological
 algebraic $K$-theory of $A$}. 
\end{definition}

\begin{remark}
 Kaygun and Khalkhali \cite{math/0606341} introduced the notion of
 $H$-equivariantly projective $A$-modules and used the exact category of
 $H$-equivariantly projective $A$-modules to define the Hopf-cyclic
 homology of $A$ .
 We may use this exact category to define an algebraic $K$-theory, which
 should be regarded as a generalization of Thomason's equivariant
 $K$-theory \cite{Thomason1987} and should be called the
 \emph{$H$-equivariant $K$-theory of $A$}.
 We note, however, that this $K$-theory is different from ours.
\end{remark}

It should be noted that our Waldhausen category can be also obtained
from the cotorsion pair $(\SemiPrj_{\Sigma},\bm{C}_{A,H}^{H})$
by using a recent work of Sarazola's \cite{1911.00613}. Sarazola's work
also suggests the existence of another interesting cotorsion
pair in $\bm{C}_{A,H}^{H}$ if we replace the structure of exact category as
follows. 

\begin{definition}
 The class of extensions in $\bm{C}_{A,H}^{H}$ that are split as
 $A$-modules is denoted by $\cE^{\mathrm{split}}$.
 We obviously obtain an exact category
 $(\bm{C}_{A,H}^{H},\cE^{\mathrm{split}})$. Let us denote this exact
 category by $\bm{C}_{A,H}^{H,\mathrm{split}}$.
\end{definition}

\begin{proposition}
 \label{cotorsion_pair_from_contractible_objects}
 Let $\Cntr_{A,H}^{H}$ be the class of $H$-equivariant $A$-modules
 that are contractible in $\lMod{H}$. 
 Then the pair $(\bm{C}_{A,H}^{H},\Cntr_{A,H}^{H})$
 is a complete cotorsion pair in $\bm{C}_{A,H}^{H,,\mathrm{split}}$.
\end{proposition}

\begin{proof}
 The completeness is obvious from the definition.

 Let us verify that $(\bm{C}_{A,H}^{H},\Cntr_{A,H}^{H})$ is a cotorsion
 pair. 
 Suppose $\Ext^{1}_{\bm{C}_{A,H}^{H,\mathrm{split}}}(M,T)=0$ for all
 $M$. We show that $T$ is contractible.
 Consider the extension
 \[
  0 \rarrow{} T \rarrow{i_{T}} C(T) \rarrow{} \Sigma(T) \rarrow{} 0.
 \]
 This is an $A$-split extension.
 Apply $\Sigma^{-1}$ and take the pullback along the map
 $T\to \Sigma^{-1}\Sigma(T)$ which induces an isomorphism in
 $\cT_{A,H}^{H}$ by Lemma \ref{Sigma} to obtain
 \[
  \xymatrix{
 0 \ar[r] & \Sigma^{-1}(T) \ar[r] & \Sigma^{-1}C(T) \ar[r] & 
 \Sigma^{-1}\Sigma(T) \ar[r] & 0  \\
 0 \ar[r] & \Sigma^{-1}(T) \ar@{=}[u] \ar[r] & E \ar[u] \ar[r] & T \ar[u]
 \ar[r] & 0.
 }
 \]
 By assumption, the bottom sequence splits and thus $T$ is a retract of
 $E$. Since these are $A$-split exact sequences, they define triangles
 in $\cT_{A,H}^{H}$. 
 Since $T\to \Sigma^{-1}\Sigma(T)$ is an isomorphism in $\cT_{A,H}^{H}$,
 so is $E\to \Sigma^{-1}C(T)$.
 Since $\Sigma^{-1}C(T)$ is contractible, so is $E$. As a retract of a
 contractible object, $T$ is contractible.

 Conversely, suppose $T$ is contractible. We have a commutative diagram
 \[
 \xymatrix{
 T \ar[dr]_{1_{T}} \ar[rr]^{1_{T}} & & T \\
 & C(T). \ar[ur]_{\varphi} & 
 }
 \]
 For an $A$-split sequence
 \[
  0 \rarrow{} F \rarrow{} E \rarrow{} T \rarrow{} 0
 \]
 take the pullback along $\varphi$ to obtain
 \[
  0 \rarrow{} F \rarrow{} \varphi^{*}(E) \rarrow{} C(T) \rarrow{} 0.
 \]
 By Lemma \ref{cone_is_perp_to_everything} this sequence splits.
 Let $s:C(T)\to \varphi^{*}(E)$ be section. Then the composition
 \[
  T \rarrow{i_{T}} C(T) \rarrow{s} \varphi^{*}(E) \rarrow{} E
 \]
 is a section of $E\to T$.
\end{proof}

\begin{remark}
 This is an analogue of Example 8.3 in Sarazola's paper
 \cite{1911.00613}. 
\end{remark}

\appendix

\end{document}